\documentclass[12pt]{amsart}

\usepackage{mathtools}
\DeclarePairedDelimiter\bra{\langle}{\rvert}
\DeclarePairedDelimiter\ket{\lvert}{\rangle}
\DeclarePairedDelimiterX\braket[2]{\langle}{\rangle}{#1 \delimsize\vert #2}

\usepackage{dsfont}
\usepackage{amsmath,amssymb,amsfonts,graphics,amsthm}
\usepackage{mathabx}
\usepackage{bbm}
\usepackage{enumerate}
\usepackage{verbatim}
\usepackage{fullpage}
\usepackage{tikz}
\usepackage{graphicx}
\usepackage{hyperref}

\newcommand{\N}{\mathbb{N}}

\newcommand{\R}{\mathbb{R}}

\newcommand{\C}{\mathbb{C}}

\newcommand{\HH}{\mathsf{H}}
\newcommand{\CH}{\overline{\mathsf{H}}}
\newcommand{\K}{\mathsf{K}}
\newcommand{\CK}{\overline{\mathsf{K}}}

\newcommand{\mc}{\mathcal}

\newcommand{\id}{\text{id}}

\newcommand{\op}[1]{\operatorname{#1}}

\newcommand{\from}{\colon}

\newtheorem{thm}{Theorem}[section]
\newtheorem{cor}[thm]{Corollary}
\newtheorem{lem}[thm]{Lemma}
\newtheorem{prop}[thm]{Proposition}

\theoremstyle{definition}
\newtheorem{defn}[thm]{Definition}

\newtheorem{rem}[thm]{Remark}

\numberwithin{equation}{section}

\begin{document}

\title[$q$-Araki-Woods Algebras]
{Complete metric approximation property for $q$-Araki-Woods Algebras} 

\date{\today}

\author{Stephen Avsec}
\address[SA]
{Chicago, Illinois, USA}
\email{stephen.avsec@gmail.com}
\author{Michael Brannan}
\address[MB]
{Department of Mathematics,
Mailstop 3368, Texas A\&M University, 
College Station, TX 77843-3368, USA}
\email{mbrannan@math.tamu.edu}
\author{Mateusz Wasilewski}
\address[MW]{Institute of Mathematics of the Polish Academy of Sciences, ul. \'{S}niadeckich 8, 00-656 Warszawa, Poland}
\email{mwasilewski@impan.pl}
\thanks{MW was partially supported by the NCN (National Centre of Science) grant  
2016/21/N/ST1/02499.}

\begin{abstract}
By adapting an ultraproduct technique of Junge and Zeng, we prove that radial completely bounded multipliers on $q$-Gaussian algebras transfer to $q$-Araki-Woods algebras. As a consequence, we establish the $w^{\ast}$-complete metric approximation property for all $q$-Araki-Woods algebras. We apply the latter result to show that the canonical ultraweakly dense C$^\ast$-subalgebras of $q$-Araki-Woods algebras are always QWEP.
\end{abstract}
\maketitle

\section{Introduction}
The study of finite approximation properties has always played a central role in the structure and classification program for operator algebras.  In the amenable setting this can be seen, for example, in the seminal work of Connes on the classification of injective factors \cite{MR0454659} and also in Elliot's classification program for simple nuclear C$^\ast$-algebras \cite{MR2415391}.  For non-amenable operator algebras, there are two approximation properties that arise as weak forms of amenability that stand out:  the \textbf{Haagerup property} and the \textbf{completely bounded approximation property}.  
These two operator algebraic properties have their roots in the deep work of Cowling, de Canni\`{e}re and Haagerup on the completely bounded multipliers of Fourier algebras and group von Neumann algebras (cf. \cite{MR520930, MR784292, MR996553}).  In the group context, amenability of a (discrete) group $G$ corresponds to the existence of an approximate identity in the Fourier algebra $A(G)$ consisting of finitely supported normalised positive definite functions.  The Haagerup property arises when one relaxes the finite support assumption and allows for an approximate unit of normalized positive definite functions that merely vanish at infinity (cf. \cite{MR718798} for the connection to group von Neumann algebras).    If one instead insists on having a finitely supported approximate unit for $A(G)$,  but allows for functions of more general type (those uniformly bounded in the completely bounded Fourier multiplier norm) this results in the fertile and robust notion of \textbf{weak amenability} (cf. \cite{MR996553}).    This latter notion has a straightforward generalization to $C^{\ast}$-algebras and von Neumann algebras, yielding the so-called (w$^*$-)completely bounded approximation property ((w$^\ast$)-CBAP). The situation is a little more subtle when translating the Haagerup property to arbitrary von Neumann algebras, and this was obtained only very recently  (cf. \cite{MR3431616} and \cite{MR3395459} for two different, but equivalent, approaches).

The w$^\ast$-CBAP has proved to be a remarkable tool in the study of non-amenable operator algebras.  Indeed, it yields a numerical invariant, called the {\it Cowling-Haagerup constant}, which was used by Cowling and Haagerup \cite{MR996553} to distinguish the group von Neumann algebras arising from lattices in the Lie groups $\text{Sp}(1,n)$.   Recently, in the breakthrough work of Ozawa and Popa (cf. \cite[Theorem 3.5]{MR2680430} and \cite{MR2393183}), the w$^\ast$-CMAP was shown to be intimately connected to several remarkable indecomposability results for finite von Neumann algebras, such as strong solidity, absence of Cartan subalgebras, primeness, and so on. 

All the results mentioned about pertain mostly to (semi)finite von Neumann algebras.  However, several recent advancements have been made in the study of type III algebras.  Most notably, the work of Isono \cite{MR3391904, MR3589353} on the structural theory of non-unimodular free quantum group factors, as well as Boutonn\'et, Houdayer and Vaes' very recent proof of strong solidity for Shlyakhtenko's free Araki-Woods factors \cite{1512.04820}.  These latter algebras constitute the very first examples of non-injective strongly solid type III factors.  Again, in the type III setting a key role is played by the w$^\ast$-CBAP, which had been established previously by Houdayer and Ricard \cite{MR2822210} for free Araki-Woods algebras, and by De Commer, Yamashita and Freslon in the free quantum group case \cite{MR3238527}. 

The present paper is concerned with the so-called $q$-Araki-Woods algebras $\Gamma_q(\HH)$, which were introduced by Hiai in \cite{MR2018229}. These (typically type III) von Neumann algebras are generated by the real parts of certain creation operators acting on a $q$-deformed Fock space $\mc F_q(\HH)$ (introduced in \cite{MR1105428}).  $\Gamma_q(\HH)$ can be viewed as a deformation of a free Araki-Woods factor depending on a parameter $q \in (-1,1)$ ($q=0$ being the undeformed case).  In many senses the $q$-Araki-Woods algebras are expected to be structurally very similar to their free, undeformed cousins.  In fact, it is even known that for and $\dim\HH < \infty$ and $|q|<< 1$ ,  $\Gamma_q(\HH)$ is isomorphic to its free cousin (cf. \cite[Theorem 4.5]{MR3312436}). However, not so much is known about these algebras in the whole admissible regime of the parameter $q$.  Let us just mention some partial results:  Very recently, advances were made on the  factoriality problem (cf. \cite{1606.04752} and \cite{1607.04027}). In many cases it is also known that $q$-Araki-Woods algebras are non-injective (cf. \cite{MR2091676}). For both properties there is really one case left open -- $q$-Araki-Woods algebras built from a two-dimensional Hilbert space  $\HH$, in which one cannot rely in any way on techniques used for $q$-Gaussian algebras, their tracial predecessors. All $q$-Araki-Woods algebras are known to be QWEP\footnote{A $C^{\ast}$-algebra is QWEP  if it is a quotient of a $C^{\ast}$-algebra possessing the weak expectation property.}  (cf. \cite{MR2200739}), and it was only recently shown by the third named author that these algebras possess the Haagerup approximation property (cf. \cite{1605.06034}).  

In this paper, our goal is to establish the w$^\ast$-CBAP for all $q$-Araki-Woods algebras.  Following Houdayer and Ricard's lead from the  free case \cite{MR2822210}, we approach this problem by trying to characterize a natural class of completely bounded maps on these algebras, called \textbf{radial multipliers}, and estimate their norms.  The classification problem for radial multipliers  appears to be hard even for small  values of $|q|$ because the known isomorphism between a $q$-Araki-Woods algebra and a free Araki-Woods factor does not carry radial multipliers to radial multipliers.  So even in this setting new techniques are crucial. In \cite{MR2822210}, the authors used the universal property of the Fock representation of the Toeplitz algebra to translate the question of computing the completely bounded norm of a radial multiplier on a free-Araki-Woods factor to an equivalent problem of computing the completely bounded norm of the {\it same} multiplier, viewed now as a radial Fourier multiplier on a free group.  In this latter setting,  one has an explicit formula (cf. \cite[Theorem 1.2]{MR2748193}) involving the trace-class norm of a Hankel matrix associated with the symbol of the multiplier.  In particular, it follows from this result that the completely bounded norms of radial multipliers on free Araki-Woods factors {\it do not} depend on the type structure of the algebra.  In the $q$-deformed setting, we conjecture that the same type-invariance for radial multipliers should hold for all $q$-Araki-Woods algebras. Unfortunately, if one tries to mimic the approach of Houdayer and Ricard in the free case, several major issues arise.  One of them is that one has to work now with the Fock representation of the $q$-deformed Toeplitz algebras, and it is an interesting open problem to settle the universality question for the Fock representation here.  In this paper we follow a different route, inspired by transference principles for multipliers.  More precisely, we develop a non-tracial version of an ultraproduct embedding theorem of Junge and Zeng for mixed $q$-Gaussian algebras \cite{1505.07852}.  Our construction (Theorem \ref{Th:Embedding}) yields a $q$-quasi-free state-preserving embedding of an arbitrary $\Gamma_q(\HH)$ into an ultraproduct of tensor products of tracial $q$-Gaussian algebras and other $q$-Araki-Woods algebras.   Using Theorem \ref{Th:Embedding}, we show that it is possible to transfer radial multipliers on (tracial) $q$-Gaussian algebras to arbitrary $q$-Araki-Woods algebras in such a way that the completely bounded norm does not increase (Theorem \ref{Th:Transference}).   Our transference result provides strong evidence towards the conjecture that radial multipliers on $q$-Araki-Woods algebras do not depend on the type structure, and we fully expect (but are unable to prove at this time) that our transference principle should be isometric and bijective.  

In any case, Theorem \ref{Th:Transference} does provide us with some new examples of completely bounded radial multipliers on $q$-Araki-Woods algebras.  These are the projections onto Wick words of a given finite length.  Upper bounds for the norms of such multipliers were obtained previously for $q$-Gaussian algebras by the first named author in \cite{1110.4918}.  These norm estimates together with the extended second quantisation functor \cite{1605.06034} turn out to be exactly what we need to establish the main result  of the paper: the w$^*$-CBAP for all $q$-Araki-Woods algebras.  In fact, just as in the free case, we obtain the completely contractive version of this property (see Section \ref{prelim} for the relevant definition):

\begin{thm}\label{Th:CMAP}
Let $\Gamma_q(\mathsf{H})$ be a $q$-Araki-Woods algebra. Then $\Gamma_q(\mathsf{H})$ has the $w^{\ast}$-complete metric approximation property.
\end{thm}

As an application of the above result, we are able to answer affirmatively a question left open by Nou \cite[Remark after Theorem 6.3]{MR2200739}, concerning whether or not the canonical w$^\ast$-dense C$^\ast$-subalgebras $\mc A_q(\HH) \subseteq \Gamma_q(\HH)$ are always QWEP;  see Corollary \ref{QWEP}.  It is our hope that Theorem \ref{Th:CMAP} will lead to a deeper  understanding of the structure of $q$-Araki-Woods algebras.  In particular, we expect this result to be a fundamental tool in the applications of deformation/rigidity tools to these algebras.

Let us conclude this section with a description of the layout of the main body of the paper.  In Section \ref{prelim} we introduce the relevant notation and background on operator spaces, von Neumann ultraproducts, and $q$-Araki-Woods algebras.  In Section \ref{embedd}, we construct our ultraproduct embedding and apply it in Section \ref{transfer} to obtain the transference principle for radial multipliers.  Finally, we present the proof of Theorem \ref{Th:CMAP} in Section \ref{proofs}.

\subsection*{Acknowledgements}
The authors are grateful to Marius Junge and \'Eric Ricard for their insights and encouragement.  Parts of this project were completed during a visit of MW to Texas A\&M University and a visit of MB and MW to L'Universit\'e de Caen.  The authors are very grateful to Gilles Pisier, \'Eric Ricard and Roland Vergnioux for facilitating these visits. 
\section{Preliminaries} \label{prelim}
\subsection{Some notation}
Throughout this paper, inner products on complex Hilbert spaces are always taken to be conjugate-linear in the left variable.  The algebraic tensor product of two complex vector spaces $V,W$ will always be denoted by $V \odot W$, and elementary tensors in $V \odot W$ will also be denoted using the symbol $\odot$.    Given a natural number $n \in \N$, we denote by $[n]$ ($[n]_{0}$) the ordered set $\{1,2, \ldots, n\}$ ($\{0,1,2,\dots,n\}$).  Given $n,d \in \N$ we will interchangeably view multi-indices $k = (k(1), k(2), \ldots, k(d)) \in [n]^d$ as functions $k:[d] \to [n]$.  Given $d \in \N$, we denote by $\mc P(d)$ the lattice of partitions of the ordered set $[d]$, and by $\mc P_2(d) \subset \mc P(d)$ the subset of pair partitions (i.e., partitions of $[d]$ into disjoint subsets (``blocks'') of size $2$).  The partial order $\le$ on $\mc P(d)$ is given by the usual refinement order on partitions, and given $\pi,\sigma \in \mc P(d)$, we denote by $\pi \vee \sigma \in \mc P(d)$ the lattice theoretic join of $\pi$ and $\sigma$ with respect to the partial order $\le$. The number of blocks of a partition $\sigma$ will be denoted by $|\sigma|$.  Finally, given a multi-index $k:[d] \to [n]$, we denote by $\ker k \in \mc P(d)$ the partition defined by level sets of $k$: that is,  $1 \le r,s \le d$ belong to the same block of $\ker k$ iff $k(r) = k(s)$.

\subsection{Operator spaces}
Some amount of the theory of operator spaces is necessary for our work; even the statement of the main result uses notions from this field. Recall that an \textbf{operator space} is a Banach space $X$ endowed with a specific choice of norms on the matricial spaces $\op{M}_{n}(X):=\op{M}_{n}\odot X$ satisfying the so-called Ruan axioms, ensuring that it comes from an isometric embedding of $X$ into$\op{B}(\HH)$, the C$^\ast$-algebra of bounded linear operators on some Hilbert space $\HH$. Given a pair of operator spaces $X$, $Y$ and a
linear map $T:X \to Y$, the {\textbf cb norm} of $T$ is given by 
\[
\|T\|_{cb}:= \sup_{n \in \N} \|\op{Id}_{n} \odot T: \op{M}_{n} \odot X \to \op{M}_n \odot Y\|. 
\]
If $\|T\|_{cb} < \infty$, we say that $T$ is  {\textbf completely bounded (cb)}.
We can now define the approximation properties that we are interested in. Let $X$ be an operator space. We say that $X$ possesses the \textbf{completely bounded approximation property} if there exists a net $(\Phi_{i})_{i \in I}$ of finite rank completely bounded maps on $X$ such that $\sup_{i\in I} \|\Phi_{i}\|_{cb} < \infty$, and $\lim_{i\in I} \|\Phi_i(x) -x\|=0$ for every $x\in X$. If we can find a net $(\Phi_{i})_{i \in I}$ such that $\|\Phi_{i}\|_{cb} \leqslant 1$ then we say that $X$ has the \textbf{complete metric approximation property}.
For a dual operator space $X$ (i.e. $X \simeq (X_{\ast})^{\ast}$ for some operator space $X_{\ast}$), there is a suitable analogue of this approximation property which takes into account this additional structure. Namely, we say that $X$ has the \textbf{$w^{\ast}$-complete metric approximation property} if there exists a net $(\Phi_i)_{i \in I}$ of finite rank  $w^{\ast}$-continuous completely bounded maps on $X$ such that
$\|\Phi_{i}\|_{cb} \leqslant 1$ for each $i \in I$, and $\lim_{i \in I} \Phi_i(x) = x$ (weak-$\ast$) for every $x \in X$.

We need to discuss two operator space structures associated with a given Hilbert space.
\begin{defn}
Let $\HH$ be a complex Hilbert space. We define the following operator space structures on $\HH$:
\begin{enumerate}[{\normalfont (i)}]
\item the \textbf{column Hilbert space} structure $\HH_{c}$ is given by the identification $\HH \simeq \op{B}(\C, \HH)$;
\item the \textbf{row Hilbert space} structure $\HH_{r}$ is given by the identification $\HH \simeq \op{B}(\CH, \C)$.
\end{enumerate}  
\end{defn}
\begin{rem}
Row (column) Hilbert spaces are \textbf{homogeneous operator spaces}, i.e. any contraction $T:\K \to \HH$ is a complete contraction $T: \K_{r} \to \HH_{r}$ ($T:\K_{c} \to \HH_{c}$) (cf. \cite[Theorem 3.4.1 and Proposition 3.4.2]{MR1793753}).
\end{rem}
These Hilbertian operator spaces will turn out to be critical for obtaining a right formulation of the non-commutative Khintchine inequalities (cf. Proposition \ref{cor:NCKhintchine} in Subsection \ref{qAraki}). 

In the theory of operator spaces there is a variety of different tensor products, analogous to tensor products of Banach spaces. There is, however, one tensor product that stands out and does not have a Banach space theoretic counterpart -- the Haagerup tensor product.
\begin{defn}
Let $X$ and $Y$ be operator spaces. We define a bilinear map $\op{M}_{n,r}(X) \times \op{M}_{r,n}(Y) \ni (x,y) \mapsto x\cdot y \in \op{M}_{n}(X \odot Y)$ to be the bilinear extension of the assignment $(A \odot x, B \odot y) \mapsto (AB, x \odot y)$. For any $z \in \op{M}_n(X \odot Y)$ we define the norm 
\[
\|z\|_{h,n} := \inf \{\|x\|\|y\|: z = x\cdot y, \ x\in\op{M}_{n,r}(X), \ y\in\op{M}_{r,n}(Y), \ r\in \N\}.
\]
This sequence of norms on the matricial spaces $\op{M}_{n}(X \odot Y)$ satisfies Ruan's axioms and therefore defines an operator space structure on $X \odot Y$, called the \textbf{Haagerup tensor product}.
The completions with respect to the norms $\|\cdot\|_{h,n}$ will be denoted $\op{M}_n(X\otimes_{h} Y)$.  For more information on the Haagerup tensor product, consult \cite[Chapter 9]{MR1793753} and \cite[Chapter 5]{MR2006539}.
\end{defn}

Later on we will need the following proposition.
\begin{prop}[Proposition 9.3.4 from \cite{MR1793753}]
Let $\K$ and $\HH$ be complex Hilbert spaces. Then the assignment $\HH \odot \CK \ni \xi \odot \eta \mapsto \ket{\xi}\bra{\eta} \in \op{K}(\K, \HH)$ (the compact operators) extends to a complete isometry $\HH_{c} \otimes_{h} \CK_{r} \simeq \op{K}(\K, \HH)$.
\end{prop}

\subsection{\texorpdfstring{$q$}{}-Araki-Woods algebras}\label{qAraki}
We present here a construction due to Hiai (cf. \cite{MR2018229}), which builds upon previous developments: $q$-Gaussian algebras of Bo\.{z}ejko and Speicher (cf. \cite{MR1463036}) and free Araki-Woods factors defined by Shlyakhtenko (cf. \cite{MR1444786}).

The starting point is a real Hilbert space $\HH_{\R}$ equipped with a continuous one parameter group of orthogonal transformations $(U_{t})_{t \in \R}$. The extension of $(U_{t})_{t \in \R}$ to a unitary group on $\HH_{\C}$, the complexification of $\HH_{\R}$, will be still denoted by $(U_{t})_{t \in \R}$. By Stone's theorem, there exists an injective, positive operator $A$ on $\HH_{\C}$ such that $U_{t} = A^{it}$. On $\HH_{\C}$ we define a new inner product $\braket{\xi}{\eta}_{U} := \braket{\xi}{\frac{2A}{1+A}\eta}$ and denote by $\HH$ the completion of $\HH_{\C}$ with respect to this inner product. Note that the norms defined by $\braket{\cdot}{\cdot}_{U}$ and $\braket{\cdot}{\cdot}$ coincide on $\HH_{\R}$. This implies that $I$, the complex conjugation on $\HH_{\C}$, is a closed operator on $\HH$ with dense domain $\HH_\C$.

Next we form the \textbf{$q$-Fock space} $\mathcal{F}_{q}(\HH)$. Since we will have to delve deeper into its structure later on, we will present the construction here. First, let us fix $q \in (-1,1)$. For any $n$ we define $P_{q}^{n}: \HH^{\odot n} \to \HH^{\odot n}$ by
\begin{equation}\label{form:qinnerproduct}
P_{q}^{n} (e_{1}\odot\dots \odot e_{n}) = \sum_{\sigma \in S_{n}} q^{i(\sigma)} e_{\sigma(1)}\odot\dots\odot e_{\sigma(n)},
\end{equation}
where $i(\sigma):=\left|\{(i,j) \in [n]^2: i<j \text{ and } \sigma(i)>\sigma(j)\}\right|$ is the number of inversions. This operator is (strictly) positive definite (cf. \cite[Proposition 1]{MR1105428}), so it defines an inner product on $\HH^{\odot n}$ by $\braket{\xi}{\eta}_{q} := \braket{\xi}{P_{q}^{n}\eta}$; the completion with respect to this inner product will be denoted by $\HH_{q}^{\otimes n}$. The $q$-Fock space is defined by the orthogonal direct sum $\mathcal{F}_{q}(\HH):=\bigoplus_{n\geqslant 0} \HH_{q}^{\otimes n}$. For our purposes, there  are two important sets of operators defined on the $q$-Fock space. For any $\xi \in \HH$ we define the \textbf{$q$-creation operator} $a_{q}^{\ast}(\xi) \in \op B(\mc F_q(\HH))$ by 
\[
a_{q}^{\ast}(\xi)(e_{1}\odot\dots\odot e_{n}) = \xi\odot e_{1}\odot \dots\odot e_{n}
\]
and the \textbf{$q$-annihilation operator} $a_{q}(\xi) = \left(a_{q}^{\ast}(\xi)\right)^{\ast}\in \op B(\mc F_q(\HH))$. It is known (cf. \cite[Remark 1.2]{MR1463036}) that 
\[\|a_q(\xi)\| = \|a_q^*(\xi)\| = \Big\{\begin{matrix}
\|\xi\|, & 0 \ge q > -1 \\
(1-q)^{-1/2}\|\xi\| & 0 < q <1. 
\end{matrix} \qquad (\xi \in \HH).\] We are now ready to define $q$-Araki-Woods algebras.
\begin{defn}
Let $(\HH_{\R}, (U_{t})_{t \in \R})$ be a real Hilbert space endowed with a one-parameter group of orthogonal transformations. Let $\HH$ be the complex Hilbert space obtained as the completion of $\HH_{\C}$ with respect to $\braket{\cdot}{\cdot}_{U}$. For any $\xi \in \HH_{\R}$ we define $s_{q}(\xi) \in \op B(\mc F_q(\HH))$ by $s_{q}(\xi)=a_{q}^{\ast}(\xi)+a_{q}(\xi)$. We define the \textbf{$q$-Araki-Woods algebra} $\Gamma_{q}(\HH)$ to be the von Neumann algebra generated by the set $\{s_q(\xi): \xi \in \HH_{\R}\}$ inside $\op{B}(\mathcal{F}_{q}(\HH))$.

In the special case $U_t=\mathds{1}$ we obtain the $q$-Gaussian algebras of Bo\.{z}ejko and Speicher and we will denote them, following the tradition, by $\Gamma_q(\HH_{\R})$ (cf. \cite[Definition 2.1]{MR1463036}). 
\end{defn}
\begin{rem}
Even though we suppress the pair $(\HH_{\R}, (U_{t})_{t \in \R})$ in the notation, we should remember how the Hilbert space $\HH$ was constructed.
\end{rem}
There is a distinguished vector $\Omega$ in $\mathcal{F}_{q}(\HH)$, called the \textbf{vacuum vector}, which is equal to $1 \in \C \simeq \HH_q^{\otimes 0} \subset \mathcal{F}_{q}(\HH)$. It is not hard to see that $\Omega$ is cyclic and separating for $\Gamma_{q}(\HH)$. In fact, one can verify that the algebraic direct sum $\bigoplus_{n\geqslant 0} \HH_{\C}^{\odot n}$ is contained in $\Gamma_{q}(\HH)\Omega$. Using the generator $A$, one can explicitly identify a big enough subset of the commutant $\Gamma_{q}(\HH)'$ for which $\Omega$ is cyclic (cf. \cite[Lemma 3.1]{MR1444786}), so $\Omega$ is also separating for $\Gamma_q(\HH)$. It follows that the normal state $\chi(\cdot) = \braket{\Omega}{\cdot \Omega}$ is faithful on $\Gamma_{q}(\HH)$ (called the \textbf{$q$-quasi-free state}) and $\mathcal{F}_{q}(\HH)$ can be identified with the GNS Hilbert space associated with $\chi$. What is more, the commutant can be identified with the version of our algebra acting on the right, but in this case not only one has to use right versions of $s_q(\xi)$ but also the real Hilbert space that one draws the vectors from needs to be changed. We record here for later use the so-called \textbf{Wick formula}, which describes the joint moments of the generators $\{s_q(\xi)\}_{\xi \in \HH_\R}$ with respect to $\chi$.

\begin{thm}[\cite{MR2018229}, \cite{MR2200739}]
For any $d \in \N$ and any $e_1, \ldots, e_d \in \HH_\R$, we have \[\chi(s_q(e_1)s_q(e_2)\ldots s_q(e_d)) = \sum_{\sigma \in \mc P_2(d)} q^{\iota(\sigma)} \prod_{(r,t) \in \sigma} \langle e_r| e_t \rangle_U,\]
where $\iota(\sigma)$ denotes the number of crossings in the pairing $\sigma \in \mc P_2(d)$, and $(r,t) \in \sigma$ indicates that $1 \le r < t \le d$ are paired together by $\sigma$.  If $d$ is odd, we interpret the above (empty) sum as $0$.
\end{thm}

 Since $ \bigoplus_{n \geqslant 0} \HH_{\C}^{\odot n} \subset \Gamma_{q}(\HH)\Omega \subset \mc F_q(\HH)$,  we are allowed to make the following definition.
\begin{defn}
Let $\xi \in \bigoplus_{n \geqslant 0} \HH_{\C}^{\odot n}$. Then there is exactly one operator $W(\xi) \in \Gamma_{q}(\HH)$, called the \textbf{Wick word} associated with $\xi$, such that $W(\xi)\Omega=\xi$.
\end{defn}
This definition will help us in constructing maps on $\Gamma_{q}(\HH)$ from operators on $\HH$.   Let us first recall a version of this construction on the level of the $q$-Fock space (cf. \cite[Lemma 1.4]{MR1463036}).
\begin{defn}
Let $T\from \K \to \HH$ be a contraction between complex Hilbert spaces. Then the assignment 
\[
\mathcal{F}_{q}(T)(e_{1}\odot\dots\odot e_{n}) = Te_{1}\odot \dots \odot Te_{n}
\]
extends to a contraction $\mathcal{F}_{q}(T)\from \mathcal{F}_{q}(\K) \to \mathcal{F}_{q}(\HH)$, called the \textbf{first quantisation} of $T$.
\end{defn}
On the level of the von Neumann algebra $\Gamma_{q}(\HH)$ it is tempting to extend the assignment 
\[
W(e_{1}\otimes \dots \otimes e_{n}) \mapsto W(Te_{1} \otimes \dots \otimes Te_{n})
\]
to a nice map on $\Gamma_{q}(\HH)$. It turns out that under a mild additional assumption on $T$  the extension exists and is a normal, unital, completely positive\footnote{From now on ``unital, completely positive'' will be abbreviated to ucp.} map. The next proposition is an extension of Theorem 2.11 from \cite{MR1463036}, which is an analogous result for $q$-Gaussian algebras.
\begin{prop}[{ \cite[Theorem 3.4]{1605.06034} }]
Let $(\K_{\R}, (V_{t})_{t \in \R})$ and $(\HH_{\R}, (U_{t})_{t \in \R})$ be real Hilbert spaces equipped with respective one-parameter orthogonal groups. Construct out of them complex Hilbert spaces $\K$ and $\HH$. Suppose that $T\from \K\to \HH$ is a contraction such that $T(\K_{\R}) \subset \HH_{\R}$ (a condition written more succinctly in the form $ITJ=T$, where $J$ and $I$ are complex conjugations on $\K_{\C}$ and $\HH_{\C}$, respectively). Then the assignment $W(e_{1}\otimes \dots \otimes e_{n}) \mapsto W(Te_{1} \otimes \dots \otimes Te_{n})$ extends to a normal ucp map $\Gamma_{q}(T)\from \Gamma_{q}(\K) \to \Gamma_{q}(\HH)$ that preserves the vacuum state. The maps $\Gamma_q(T)$ is called the \textbf{second quantisation} of $T$.
\end{prop}
\begin{rem}
The analogous result in the free case (i.e. $q=0$) was crucial in the proof of the w$^\ast$-complete metric approximation property for the free Araki-Woods factors in \cite{MR2822210}. 
\end{rem}

\subsection{Wick formula and Nou's non-commutative Khintchine inequality}\label{qAraki-Wick}
To fulfill the purpose of this paper, that is to prove the w$^\ast$-complete metric approximation property for the $q$-Araki-Woods algebras, we need to expand our knowledge of the Wick words. Let us start with the celebrated Wick formula.  The proof of the following result can be found in \cite[Proposition 2.7]{MR1463036} in the tracial case.  The general case follows along the same lines.  See also \cite[Lemma 3.2]{MR2822210}.
\begin{prop}[Wick formula]\label{prop:Wick}
Suppose that $e_1,\dots, e_n \in \HH_{\C}$. Then
\begin{equation}\label{form:Wickformula}
W(e_1\odot\dots\odot e_n) = \sum_{k=0}^{n} \sum_{i_1,\dots,i_{n-k},j_{n-k+1},\dots,j_n} a_q^{\ast}(e_{i_1})\dots a_q^{\ast}(e_{i_{n-k}})a_q(I e_{j_{n-k+1}})\dots a_q(I e_{j_n}) q^{i(I_1,I_2)},
\end{equation}
where $I_1 = \{i_1<\dots< i_{n-k}\}$ and $I_2 = \{j_{n-k+1}<\dots<j_n\}$ form a partition of the set $[n]$ and $i(I_1,I_2) = \sum_{l=1}^{n-k} (i_{l}-l)$ is the number of inversion of the permutation defined by $I_1$ and $I_2$. In particular, we have $W(e) = s_q(e)$ for any $e \in \HH_\R$.
\end{prop}
We will be concerned with the subspaces $\Gamma_{q}^{n}(\HH)$ of $\Gamma_{q}(\HH)$ spanned by the sets $\{W(\xi): \xi \in \HH_{\C}^{\odot n}\}$; elements of these subspaces will be called {\textbf Wick words of length $n$}.  We will also denote by $\widetilde{\Gamma}_q(\HH) \subseteq \Gamma_{q}(\HH)$ the (non-closed) linear span of $\big(\Gamma_{q}^{n}(\HH)\big)_{n \in \N_0}$.  Note that $\widetilde{\Gamma}_q(\HH)$ is a w$^\ast$-dense $\ast$-subalgebra of $\Gamma_q(\HH)$, called the {\textbf algebra of Wick words}. Note that if $\xi=e_1\odot\dots\odot e_n$, where $e_1,\dots,e_n \in \HH_{\R}$ then $W(\xi) - s_q(e_1)\dots s_q(e_n)$ is a sum of Wick words of length strictly smaller than $n$, so inductively one can show that $\widetilde{\Gamma}_q(\HH)$ is the same as the $\ast$-algebra generated by $\{s_q(\xi):\xi \in \HH_{\R}\}$. Let now $(e_i)_{i \in I}$ be a fixed orthonormal basis for $\HH_\R$.  Then the algebra of Wick words $\widetilde{\Gamma}_q(\HH)$ is $\ast$-isomorphic to the $\ast$-algebra of non-commutative polynomials $\C \langle (X_i)_{i \in I} \ | \ X_i=X_i^*\rangle$. The isomorphism in this case is given by $X_i \mapsto s_q(e_i) = W(e_i)$.  See \cite[Remark after Lemma 3.2]{MR2200739} for details. At times we will also need to consider the C$^\ast$-completion $\mc A_q(\HH)$ of $\widetilde{\Gamma}_q(\HH)$.  The most important part of the proof of the main theorem is providing an estimate (which must grow at most polynomially in $n$) for the cb norm of the projection from $\widetilde{\Gamma}_q(\HH)$ onto $\Gamma_{q}^{n}(\HH)$.   Therefore we need to understand the operator space structure of these spaces. This will be acccomplished by reformulating the Wick formula so that it is more amenable to operator space theoretic techniques, following Nou's lead (cf. \cite{MR2091676}). We first define some relevant maps.
\begin{defn}
Let $\HH$ be a complex Hilbert space coming from a pair $(\HH_{\R}, (U_{t})_{t \in \R})$. We define maps $\mathcal{I}$, $U$ and $S$ on the algebraic direct sum $\bigoplus_{n \geqslant 0} \HH_{\C}^{\odot n}$ by
\begin{enumerate}[{\normalfont (i)}]
\item $\mathcal{I}(e_{1}\odot\dots\odot e_{n}):=Ie_{1}\odot \dots\odot Ie_{n}$;
\item $U(e_{1}\odot \dots \odot e_{n}) = e_{n} \odot \dots\odot e_{1}$;
\item $S = \mathcal{I}U$.
\end{enumerate}
The antilinear map $\mathcal{I}$ is a natural extension of the complex conjugation on $\HH_{\C}$, thereby it should be really viewed as a closed linear operator from $\mathcal{F}_{q}(\HH)$ to $\mathcal{F}_{q}(\CH)$ mapping $e_{1}\otimes\dots\otimes e_{n}$ to $\overline{I e_{1}} \otimes \dots \otimes \overline{Ie_{n}}$. The flip map $U$ actually extends to a unitary on $\mathcal{F}_{q}(\HH)$. The last map, $S$, is a conjugation relevant to the Tomita-Takesaki theory. For future reference, let us point out that the {\textbf modular automorphism group} $(\sigma_t)_{t \in \R}$ associated to the $q$-quasi-free state $\chi$ was computed in \cite{MR1444786, MR2018229}, and is given by
\[
\sigma_t(s_q(\xi)) = s_q(U_{-t}\xi) = s_q(A^{-it}\xi) \qquad (\xi \in \HH_\C). 
\]  
\end{defn}
We still need two more maps for our reformulation of the Wick formula.
\begin{defn}
Fix $k \in \N_0$ and $n \in \N$ such that $0 \le k \leqslant n$. We define the map $R_{n,k}^{\ast}\from \HH_{q}^{\otimes n} \to \HH_{q}^{\otimes (n-k)} \otimes_h \HH_{q}^{\otimes k}$ by specifying its values on a dense subspace:
\[
R_{n,k}^{\ast}(e_{1}\odot \dots \odot e_{n}):= \sum_{i_1,\dots,i_{n-k},j_{n-k+1},\dots,j_n} q^{i(I_1, I_{2})} (e_{i_{1}}\odot \dots \odot e_{i_{n-k}}) \otimes_{h} (e_{j_{n-k+1}}\odot\dots\odot e_{j_{n}}).
\]
We also define $U_{n,k}\from \left(\HH_{q}^{\otimes (n-k)}\right)_{c} \otimes_{h} \left(\CH_{q}^{\otimes k}\right)_{r} \to \op{B}(\mathcal{F}_{q}(\HH))$ by
\[
U_{n,k}((e_{1}\odot \dots\odot e_{n-k})\otimes_{h}(\overline{e_{n-k+1}}\odot\dots\odot \overline{e_{n}})):= a_{q}^{\ast}(e_{1})\dots a_{q}^{\ast}(e_{n-k}) a_{q}(e_{n-k+1})\dots a_q(e_{n}).
\]
\end{defn}
\begin{rem}
Nou proved (cf. \cite[Lemma 4 and Corollary 1]{MR2091676}) that $\|U_{n,k}\|_{cb} \leqslant C(q)$, where $C(q) = \prod_{n=1}^\infty (1-q^n)^{-1}.$  This constant $C(q)$ will appear throughout the paper.
\end{rem}

We are now ready to state the reformulated Wick formula and the corresponding Khintchine inequality.

\begin{prop}
For any $\xi \in \HH_{\C}^{\odot n}$ we have $W(\xi) = \sum_{k=0}^{n} U_{n,k} (\mathds{1}_{n-k} \odot \mathcal{I}) R_{n,k}^{\ast}(\xi) $, where $\mathds{1}_{n-k}$ is the identity map on $\HH_{\C}^{\odot (n-k)}$.
\end{prop}
\begin{cor}[{\cite[Theorem 3]{MR2091676}}]\label{cor:NCKhintchine}
Let $\K$ be a Hilbert space. If $\xi \in \op{B}(\K) \odot \HH_{\C}^{\odot n}$ then
\begin{align}
\max_{0\leqslant k \leqslant n} \|\left(\op{Id}\odot (\mathds{1}_{n-k} \odot \mathcal{I}) R_{n,k}^{\ast}\right)(\xi)\| &\leqslant \|(\op{Id}\odot W)(\xi)\| \label{minorization} \\ 
\|(\op{Id}\odot W)(\xi)\|&\leqslant C(q)(n+1) \max_{0\leqslant k \leqslant n} \|\left(\op{Id}\odot (\mathds{1}_{n-k} \odot \mathcal{I}) R_{n,k}^{\ast}\right)(\xi)\| \label{majorization}.
\end{align}
The norm $\|(\op{Id}\odot W)(\xi)\|$ is computed in $\op{B}(\K) \otimes_{\op{min}} \Gamma_{q}(\HH)$, and the other norms are computed in $\op{B}(\K)\otimes_{\op{min}} \left(\HH_{q}^{\otimes (n-k)}\right)_{c} \otimes_{h} \left(\CH_{q}^{\otimes k}\right)_{r}$.
\end{cor}
\begin{proof}
Inequality \eqref{majorization} follows from the Wick formula, complete boundedness of $U_{n,k}$ and the triangle inequality, as in the proof of Theorem $1$ in \cite{MR2091676}. The proof of \eqref{minorization} is also a repetition of the argument in Nou's paper.
\end{proof}

\subsection{Radial multipliers on \texorpdfstring{$q$}{}-Araki-Woods algebras}

In this paper, we will be primarily interested in a special class of completely bounded linear maps on $q$-Araki-Woods algebras, called radial multipliers.  In the following, we fix an arbitrary $q$-Araki-Woods algebra $\Gamma_q(\HH)$. 

\begin{defn}
Let $\varphi: \N_0 \to \C$ be a bounded function.  The (w$^\ast$-densely defined) linear map $\mathsf{m}_{\varphi}: \widetilde{\Gamma}_q(\HH) \to \widetilde{\Gamma}_q(\HH)$ given by \[\mathsf{m}_{\varphi} (W(\xi)) = \varphi(m) W(\xi) \qquad (\xi \in (\HH_\C)^{\odot m})
\]
is called the {\textbf radial multiplier} with symbol $\varphi$.  If $\mathsf{m}_{\varphi}$ extends to a completely bounded map $\mathsf{m}_{\varphi}:\mc A_q(\HH) \to \mc A_q(\HH)$, we call $\mathsf{m}_{\varphi}$ a {\textbf completely bounded radial multiplier} on $\Gamma_q(\HH)$.
\end{defn}

\begin{rem}
It is an easy exercise to see that completely bounded radial multipliers $\mathsf{m}_\varphi$ automatically extend (uniquely) to normal maps on $\Gamma_q(\HH) =\mc A_q(\HH)''$ with the same cb norm (cf. \cite[Lemma 3.4]{MR2822210}). It comes from the fact that one can work on the level of the predual of $\Gamma_q(\HH)$ and there the subspace generated by $\mc A_q(\HH)$ is norm-dense.
\end{rem}

In the course of the proof of the complete metric approximation property for $q$-Araki-Woods algebras we will need the following result obtained by the first-named author.
\begin{thm}[{\cite[Proposition 3.3 and the remark following it]{1110.4918}}]\label{Th:Polynomial bound}
Let $\mathsf{H}_{\R}$ be a real Hilbert space and let $\Gamma_q(\mathsf{H}_{\R})$ be the $q$-Gaussian algebra associated with it. Fix $n \in \N$ and and consider the radial multiplier $\mathsf{m}_{\varphi}$ associated to the Kronecker delta symbol $\varphi_n(k) = \delta_n(k) = \delta_{n,k}$. Then $\mathsf{m}_{\varphi}$ is a cb radial multiplier and corresponds to the projection $P_n$ of $\Gamma_q(\mathsf{H}_{\R})$ onto the ultraweakly closed span of $\{W(\xi): \xi \in \mathsf{H}^{\odot n}\}$.  Moreover, we have $\|\mathsf{m}_{\varphi_n}\|_{cb} \leqslant C(q)^2 (n+1)^2$.
\end{thm}
\subsection{Ultraproducts of von Neumann algebras}
In this subsection we will mostly follow \cite{MR3198856}.  Ultraproducts of von Neumann algebras are very useful, e.g. in the study of central sequences in connection with property $\Gamma$. The original construction was applicable only in the case of tracial algebras. The main difference in the type III case is that there are two different notions of ultraproducts, each havings its own virtues.

We start with a definition due to Ocneanu \cite{MR807949}, which is closer to the ultraproduct of tracial von Neumann algebras. We fix a sequence $(\mathsf{M}_{n}, \varphi_{n})_{n \in \N}$ of von Neumann algebras equipped with normal faithful states, and a non-principal ultrafilter $\omega$ on $\N$. Recall that if all the states were tracial, the ultraproduct would be defined as the direct product 
\[
\ell^{\infty}(\N, \mathsf{M}_{n}):=\{(x_{n})\in\prod_{n \in \N} \mathsf{M}_{n}: \sup_{n \in \N} \|x_n\| <\infty\}
\]
quotiented by the ideal of $L^{2}$-null sequences, i.e. sequences $(x_{n}) \in \ell^{\infty}(\N, \mathsf{M}_{n})$ such that $\lim_{n \to \omega} \varphi_{n}(x_{n}^{\ast}x_{n})=0$. The problem in the non-tracial case is that this subspace is just a left ideal and there is no reason why we should prefer $\lim_{n \to \omega} \varphi_{n}(x_{n}^{\ast}x_{n})$ to $\lim_{n \to \omega} \varphi(x_{n} x_{n}^{\ast})$. This little nuisance can be taken care of by defining $\|x\|^{\#}_{\varphi}:= \left(\varphi(x^{\ast}x+xx^{\ast})\right)^{\frac{1}{2}}$ and working with the condition $\lim_{n \to \omega} \|x_{n}\|^{\#}_{\varphi_n}=0$ instead. This, unfortunately, gives rise to another problem -- the subspace 
\[\mathsf{I}_{\omega}(\mathsf{M}_{n}, \varphi_{n}):=\{(x_{n}) \in \ell^{\infty}(\N, \mathsf{M}_{n}): \lim_{n\to \omega} \|x_{n}\|^{\#}_{\varphi_{n}}=0\}
\] is still not an ideal. We need to find the largest subalgebra inside $\ell^{\infty}(\N, \mathsf{M}_{n})$ in which $\mathsf{I}_{\omega}(\mathsf{M}_{n}, \varphi_{n})$ is an ideal. This leads us to the next definition.
\begin{defn}
Let $(\mathsf{M}_{n}, \varphi_{n})$ be a sequence of von Neumann algebras equipped with normal faithful states. Define
\[
\mathcal{M}^{\omega}(\mathsf{M}_{n}, \varphi_n):=\{(x_{n})_{n\in \N} \in \ell^{\infty}(\N, \mathsf{M}_{n}): (x_n)\mathsf{I}_{\omega} \subset \mathsf{I}_{\omega}, \ \mathsf{I}_{\omega}(x_n) \subset \mathsf{I}_{\omega}\}.
\]
Then $\mathcal{M}^{\omega}(\mathsf{M}_{n}, \varphi_n)$ is a $C^{\ast}$-algebra in which $\mathsf{I}_{\omega}(\mathsf{M}_{n}, \varphi_{n})$ is a closed ideal. Therefore we can form the quotient 
\[
(\mathsf{M}_n, \varphi_n)^{\omega}:= \mathcal{M}^{\omega}(\mathsf{M}_{n}, \varphi_n) \slash \mathsf{I}_{\omega}(\mathsf{M}_{n}, \varphi_{n})
\]
which is, \emph{a priori}, a $C^{\ast}$-algebra but actually turns out to be a von Neumann algebra (cf. \cite[Proposition on page 32]{MR807949}), called the \textbf{Ocneanu ultraproduct} of the sequence $(\mathsf{M}_{n}, \varphi_{n})_{n \in \N}$. The image of a sequence $(x_{n})_{n \in \N} \in \mathcal{M}^{\omega}(\mathsf{M}_n, \varphi_n)$ in the quotient algebra $(\mathsf{M}_{n},\varphi_n)^{\omega}$ will be denoted by $(x_{n})^{\omega}$.
\end{defn}
\begin{rem}
The \textbf{ultraproduct state} $(\varphi_{n})^{\omega}( (x_{n})^{\omega}):= \lim_{n \to \omega} \varphi_{n}(x_{n})$ is a normal faithful state on $(\mathsf{M}_n, \varphi_n)^{\omega}$.
\end{rem}
Despite being a natural generalisation of the tracial ultraproduct, the Ocneanu ultraproduct suffers from being inadequate for the purpose of non-commutative integration. One particular problem is that the Banach space ultraproduct of preduals is usually bigger than the predual of the Ocneanu ultraproduct.

There is a different construction that, as shown in \cite{MR1926043}, interacts nicely with ultraproducts of non-commutative $L^{p}$-spaces. Once again, we start from a sequence $(\mathsf{M}_{n}, \varphi_{n})_{n \in \N}$ of von Neumann algebras endowed with normal faithful states. Using the GNS construction, we view $\mathsf{M}_{n} \subset \op{B}(\mathsf{H}_{n})$. Let $(\mathsf{M}_{n})_{\omega}$ denote the Banach space ultraproduct of the sequence $(\mathsf{M}_{n})_{n \in \N}$, which is a $C^{\ast}$-algebra. Let $(\mathsf{H}_{n})_{\omega}$ be the ultraproduct of the corresponding GNS Hilbert spaces. Then we can view $(\mathsf{M}_{n})_{\omega}$ as acting on $(\mathsf{H}_{n})_{\omega}$ via
\begin{equation}\label{form:diagrep}
(x_{n})_{\omega}(\xi_n)_{\omega} := (x_n \xi_n)_{\omega}.
\end{equation}
It is not hard to see that this is well defined (by the joint continuity of the map $\op{B}(\mathsf{H}) \times \mathsf{H} \ni (x,\xi) \mapsto x\xi \in \mathsf{H}$). 
\begin{defn}
Let $(\mathsf{M}_n, \varphi_n)_{n \in \N}$ be a sequence of von Neumann algebras equipped with normal faithful states, represented faithfully on the GNS Hilbert spaces, i.e. $\mathsf{M}_{n} \subset \op{B}(\mathsf{H}_n)$. The \textbf{Raynaud ultraproduct} is defined as the weak closure inside $\op{B}((H_n)_{\omega})$ of the image of the natural diagonal representation \eqref{form:diagrep} of the $C^{\ast}$-ultraproduct $(\mathsf{M}_{n})_{\omega}$ on $(\mathsf{H}_n)_{\omega}$; it is denoted by $\prod^{\omega}(\mathsf{M}_n, \varphi_n)$.  
\end{defn}
\begin{rem}
The ultraproduct state on the Raynaud ultraproduct, which is a vector state induced by the ultraproduct of the cyclic vectors for the GNS representations of algebras $\mathsf{M}_n$, denoted by $(\varphi_n)_{\omega}$, is generally not faithful.
\end{rem}
There is a nice relationship between the two constructions which is summarised in the following theorem.
\begin{thm}[{\cite[Theorem 3.7]{MR3198856}}]\label{thm:HaagerupAndo}
Let $(\mathsf{M}_n, \varphi_n)_{n \in \N}$ be a sequence of von Neumann algebras equipped with normal faithful states. Let $\mathsf{H}_n:=L^{2}(\mathsf{M}_n, \varphi_n)$ be the GNS-Hilbert space associated with the state $\varphi_n$ on $\mathsf{M}_n$, so we have $\prod^{\omega}(\mathsf{M}_n, \varphi_n) \subset \op{B}((\mathsf{H}_n)_{\omega})$. Let $\mathsf{M}^{\omega} := (\mathsf{M}_n, \varphi_n)^{\omega}$ and $\varphi^{\omega}  := (\varphi_n)^{\omega}$. Define a map $w\from L^{2}(\mathsf{M}^{\omega},\varphi^{\omega}) \hookrightarrow (\mathsf{H}_n)_{\omega}$ from the GNS-Hilbert space of $(\mathsf{M}^{\omega}, \varphi^{\omega})$ given by
\[
w\left( (x_n)^{\omega}(\xi_{\varphi^{\omega}})\right) := (x_n \xi_{\varphi_n})_{\omega},
\]
where $\xi$ (with an appropriate subscript) is the cyclic vector coming from the GNS construction. Then $w$ is an isometry and $w^{\ast} \left(\prod^{\omega}(\mathsf{M}_n, \varphi_n)\right)w = \mathsf{M}^{\omega}$.
\end{thm}
\begin{rem}
Let $p$ be the support projection of the ultraproduct state $(\varphi_n)_{\omega}$ on the Raynaud ultraproduct $\prod^{\omega}(\mathsf{M}_n, \varphi_n)$. Then $\mathsf{M}^{\omega}$ is $\ast$-isomorphic to the corner $p\left(\prod^{\omega}(\mathsf{M}_n, \varphi_n)\right)p$ of the Raynaud ultraproduct (cf. \cite[Proposition 3.15]{MR3198856}).
\end{rem}
We would now like to describe a useful theorem from \cite{MR2200739} concerning embeddings into ultraproducts.
\begin{thm}[{\cite[Theorem 4.3]{MR2200739}}]\label{thm:Nouembedding}
Let $(\mathsf{N}, \psi)$ and $(\mathsf{M}_n, \varphi_n)_{n \in \N}$ be von Neumann algebras equipped with normal faithful states. Let $\omega$ be a non-principal ultrafilter on $\N$ and let $\prod^{\omega}(\mathsf{M}_n, \varphi_n)$ be the Raynaud ultraproduct. Let $(\sigma_t^n)_{t \in \R}$ denote the modular group of $\varphi_n$. Let $p\in \prod^{\omega}(\mathsf{M}_n, \varphi_n)$ denote the support of the ultraproduct state $(\varphi_n)_{\omega}$. Suppose that $\widetilde{\mathsf{N}} \subset \mathsf{N}$ is a weak$^{\ast}$-dense $\ast$-subalgebra of $\mathsf{N}$ and we are given a $\ast$-homomorphism 
\[
\Phi\from \widetilde{\mathsf{N}} \to \prod^{\omega}(\mathsf{M}_n, \varphi_n).
\]
Assume that $\Phi$ satisfies the following conditions:
\begin{enumerate}[{\normalfont (i)}]
\item It is state preserving, i.e. $ (\varphi_n)_{\omega} \left(\Phi(x)\right) = \psi(x)$ for any $x \in \widetilde{\mathsf{N}}$;
\item\label{cond:analyticity} For any $x \in \Phi(\widetilde{\mathsf{N}})$ there is a representative $(x_n)_{n\in\N} \in \ell^{\infty}(\N, \mathsf{M}_n)$ such that $x_n$ is analytic for $(\sigma_t^n)_{t \in \R}$ and the sequence $(\sigma_{-i}^n (x_n))_{n \in \N}$ is bounded (cf. \cite[Lemma 4.1]{MR2200739}).
\item For all $t\in \R$ and for all $y = (y_n)_{\omega} \in \Phi(\widetilde{\mathsf{N}})$, irrespective of the choice of the representative $(y_n)_{n\in\N} \in \ell^{\infty}(\N, \mathsf{M}_n)$, we have
\[
p (\sigma_t^n(y_n))_{\omega} p \in p\mathsf{B}p,
\]
where $\mathsf{B}$ is the w$^{\ast}$-closure of $\Phi(\widetilde{\mathsf{N}})$. 
\end{enumerate}
Then the map $\Theta:=p\Phi p: \widetilde{\mathsf{N}} \to p\left(\prod^{\omega}(\mathsf{M}_n, \varphi_n)\right)p$ is a state-preserving $\ast$-homomorphism that can be extended to a normal $\ast$-isomorphism from $\mathsf{N}$ onto $p\mathsf{B}p$. Moreover, there exists a normal, state-preserving conditional expectation from $\prod^{\omega}(\mathsf{M}_n, \varphi_n)$ onto $\Theta(\mathsf{N})$.
\end{thm}
\begin{rem}
From the Remark following Theorem \ref{thm:HaagerupAndo} we deduce that the image of $\Theta$ is actually contained in an isomorphic copy of Ocneanu ultraproduct.
\end{rem}
This theorem will be our most important ally in the next section of the paper. We will exploit the fact that it connects nicely the Ocneanu ultraproduct and the Raynaud ultraproduct, allowing us to resort to whichever one we prefer; both will be useful for us.

We end this section with a simple remark strengthening the connection between Theorem \ref{thm:Nouembedding} and the Ocneanu ultraproduct. Suppose that $(x_n)_{n \in \N} \in \ell^{\infty}(\N,\mathsf{M}_n)$ is a representative of an element $x\in (\mathsf{M}_n)_{\omega}$ such that the sequence $(\sigma_{-i}^n(x_n))_{n \in \N}$ is bounded. Then the sequence $(x_n)_{n \in \N}$ belongs to $\mathcal{M}^{\omega}(\mathsf{M}_n, \varphi_n)$, so it defines an element of the Ocneanu ultraproduct. Indeed, suppose that $(y_n) \in \mathsf{I}_{\omega}(\mathsf{M}_n, \varphi_n)$. We would like to check that $\lim_{n \to \omega} \|x_n y_n\|_{\varphi_n}^{\#} = 0$. It boils down to checking that $\lim_{n \to \omega} \varphi_{n} (y_n^{\ast}x_n^{\ast}x_n y_n)=0$ and $\lim_{n \to \omega} \varphi_n (x_n y_n y_n^{\ast} x_n^{\ast}) = 0$. The first equality is easy to verify because $y_n^{\ast} x_n^{\ast}x_n y_n \leqslant \|x_n\|^2 y_n^{\ast} y_n$ and the sequence $(x_n)_{n \in \N}$ is bounded. For the second one we will use the KMS condition:
\[
\varphi_n (x_n y_n y_n^{\ast} x_n^{\ast}) = \varphi_n( y_n y_n^{\ast} x_n^{\ast} \sigma_{-i}^n(x_n)).
\]
Note that $z_n:= x_n^{\ast} \sigma_{-i}^n(x_n)$ is a bounded sequence. If we denote $u_n = \sqrt{y_n y_n^{\ast}}$ then we have to bound $\varphi(u_n^2 y_n)$. By the Cauchy-Schwarz inequality we get
\[
|\varphi_n(u_n (u_n z_n)) \leqslant \varphi_n (u_n^2) \varphi_n(z_n^{\ast} u_n^2 z_n).
\]
By assumption we have $\lim_{n\to \omega} \varphi_n(u_n^2) = \lim_{n \to \omega} \varphi_n(y_n y_n^{\ast}) = 0$. The second term can be bounded above by the norm $\|z_n^{\ast} u_n^2 z_n\|$ that is bounded, so the product converges to zero.

\section{An ultraproduct embedding for \texorpdfstring{$q$}{}-Araki-Woods algebras} \label{embedd}
In this section we prove a result which shows that an arbitary $q$-Araki-Woods algebra embeds in a state preserving way into an ultraproduct of tensor products of $q$-Gaussian algebras and $q$-Araki-Woods algebras.  This result will be key to our establishment of a transference principle for completely bounded radial multipliers in the following section.

We begin with some notation.  Let $\Gamma_q(\mathsf{H})$ be a fixed $q$-Araki-Woods algebra for some $q\in (-1,1)$, and write $q=q_0 q_1$ for some $|q|<q_0<1$. For any $m\in \N$, we let $\Gamma_{q_0}(\R^m)$ be a $q$-Gaussian algebra and  $\Gamma_{q_1}(\mathsf{H} \otimes \C^m)$ be a $q$-Araki-Woods algebra, where the inner product on $\mathsf{H} \otimes \C^m$ is the tensor product of the given deformed inner product on $\mathsf{H}$ and the non-deformed one on $\C^m$.  In other words, if $(U_t)_t \curvearrowright \HH_\R$ is the orthogonal group associated to $\Gamma_q(\mathsf{H})$, then $(U_t \otimes 1)_t \curvearrowright \HH_\R \otimes \R^m$ is the orthogonal group associated to $\Gamma_{q_1}(\mathsf{H} \otimes \C^m)$.  Denote by $\chi,\chi_{0,m}$ and $\chi_{1,m}$ the $q$-quasi-free states on $\Gamma_q(\mathsf{H}), \Gamma_{q_0}(\R^m)$ and $\Gamma_{q_1}(\mathsf{H} \otimes \C^m)$, respectively. For each $m$, fix an orthonormal basis $(e_1,\dots,e_m)$ of $\R^m$ and define
\[
u_m(\xi) := \frac{1}{\sqrt{m}} \sum_{k=1}^{m} W(e_k) \otimes W(\xi \odot e_k) \in \Gamma_{q_0}(\R^m) \overline{\otimes} \Gamma_{q_1}(\mathsf{H} \otimes \C^m) \qquad (\xi \in \HH_\C).
\]
Finally, we fix a non-principal ultrafilter $\omega$ on $\N$, form the corresponding (Raynaud) ultraproduct \[\mathsf A = \prod^{\omega}\Big( \Gamma_{q_0}(\R^m) \overline{\otimes} \Gamma_{q_1}(\mathsf{H} \otimes \C^m), \chi_{0,m} \otimes \chi_{1,m}\Big),\] and let $p \in \mathsf A$ be the support of the ultraproduct state $(\chi_{0,m} \otimes \chi_{1,m})_\omega$.

With the above notation fixed, we can now state our embedding result.

\begin{thm}\label{Th:Embedding}
\begin{enumerate}
\item The mapping 
\[
W(\xi) \mapsto (u_m(\xi))_\omega \in \mathsf A \qquad (\xi \in \mathsf{H}_\C)
\]
extends uniquely to a state-preserving $\ast$-homomorphism $\pi_\omega:(\widetilde{\Gamma}_q(\HH), \chi) \to (\mathsf A, (\chi_{0,m} \otimes \chi_{1,m})_\omega)$.  
\item The map $\Theta:=p\pi_\omega(\cdot) p: \widetilde{\Gamma}_q(\HH) \to p\mathsf A p$ extends to a normal state-preserving $\ast$-isomorphism 
\[
\Theta: \Gamma_q(\HH) \to \Theta(\Gamma_q(\HH)) \subseteq p\mathsf  A p.
\] 
Moreover, $\Theta(\Gamma_q(\HH))$ is the range of a normal state-preserving conditional expectation $E: \mathsf  A \to \Theta(\Gamma_q(\HH))$.
\end{enumerate}
\end{thm}

\begin{proof}
(1). Recall that the algebra of Wick words is $\ast$-isomorphic to the $\ast$-algebra of non-commutative polynomials, so any $\ast$-homomorphism $\pi_\omega: \widetilde{\Gamma}_q(\HH) \to \mathsf  A$ is uniquely determined by specifying the images $(\pi_{\omega}(W(e_i)))_{i \in I} \subset \mathsf  A$.  Thus to conclude that the claimed $\pi_\omega$ exists and is well-defined, we just need to check that each sequence $(u_m(\xi))_{m \in \N}$ ($\xi \in \HH_\C$) is norm-bounded and hence defines an element $(u_m(\xi))_{\omega} \in \mathsf A$.
To this end, we apply (the $n=1$ version of) Corollary \ref{cor:NCKhintchine} with coefficients $W(\xi \odot e_k) \in \op{B}(\mathsf{K})=\op{B}(\mathcal{F}_{q_1}(\HH \otimes \C^m))$ (see also \cite[Page 17]{MR2091676}) to conclude that 
\begin{align*}
\|u_m(\xi)\| &\le 2(1-q_0)^{\frac{-1}{2}}m^{\frac{-1}{2}} \max\Big\{\Big\|\sum_{k=1}^m W(\xi \odot e_k)^*W(\xi \odot e_k)\Big\|^{\frac12}, \Big\|\sum_{k=1}^m W(\xi \odot e_k)W(\xi \odot e_k)^*\Big\|^{\frac12}\Big\}\\
& \le 2(1-q_0)^{\frac{-1}{2}}\|W(\xi \odot e_1)\|.  
\end{align*} 
Finally we check that $\pi_\omega$ is state-preserving.   By linearity, it suffices to show that for any $ d \in \N$ and  $\xi_1, \ldots, \xi_d \in \HH_\R$, we have 
\[\lim_{m \to \infty} (\chi_{0,m} \otimes \chi_{1,m})\big(u_{m}(\xi_1) \cdot \ldots \cdot u_m(\xi_d)\big) = \chi(W(\xi_1) \cdot \ldots \cdot W(\xi_d)).\]
Fixing $m$ and considering the terms on the left-hand side above, we have 
\begin{align*}
& (\chi_{0,m} \otimes \chi_{1,m})\big(u_{m}(\xi_1) \cdot \ldots \cdot u_m(\xi_d)\big)   \\
&=m^{-d/2} \sum_{k:[d] \to [m]} \chi_{0,m}(W(e_{k(1)}) \cdot \ldots \cdot W(e_{k(d)})) \chi_{1,m}(W(\xi_1 \odot e_{k(1)}) \cdot \ldots \cdot W(\xi_d \odot e_{k(d)}))\\
&=m^{-d/2} \sum_{k:[d] \to [m]} \Big(\sum_{\substack{\sigma \in \mc  P_2(d) \\
\ker k \ge \sigma}}q_0^{\iota(\sigma)}\Big) \Big( \sum_{\sigma' \in \mc P_2(d)} q_1^{\iota(\sigma')} \prod_{(r,t) \in \sigma'} \langle \xi_r \odot e_{k(r)}| \xi_t \odot e_{k(t)} \rangle_U \Big)\\
&=m^{-d/2} \sum_{k:[d] \to [m]} \Big(\sum_{\substack{\sigma \in \mc  P_2(d) \\
\ker k \ge \sigma}}q_0^{\iota(\sigma)}\Big) \Big( \sum_{\substack{\sigma' \in \mc P_2(d)\\\ker k \ge \sigma'}} q_1^{\iota(\sigma')} \prod_{(r,t) \in \sigma'} \langle \xi_r| \xi_t  \rangle_U \Big) \\
&= \sum_{\substack{\sigma,\sigma' \in \mc P_2(d)}} q_0^{\iota(\sigma)} q_1^{\iota(\sigma')} \prod_{(r,t) \in \sigma'} \langle \xi_r| \xi_t  \rangle_U   \sum_{\substack{k:[d]\to [m]\\ \ker k \ge \sigma, \ker k \ge \sigma '}} m^{-d/2} \\
&=\sum_{\substack{\sigma,\sigma' \in \mc P_2(d)}} q_0^{\iota(\sigma)} q_1^{\iota(\sigma')} \prod_{(r,t) \in \sigma'} \langle \xi_r| \xi_t  \rangle_U   m^{-d/2 + |\sigma \vee \sigma'|}.
\end{align*} Since 
\[\lim_{m\to \infty} m^{-d/2 + |\sigma \vee \sigma'|} = \delta_{\sigma,\sigma'}  \qquad (\sigma, \sigma ' \in \mc P_2(d)),\]
we conclude that 
\[\lim_{m \to \infty} (\chi_{0,m} \otimes \chi_{1,m})\big(u_{m}(\xi_1) \cdot \ldots \cdot u_m(\xi_d)\big) = \sum_{\substack{\sigma \in \mc P_2(d)}} q^{\iota(\sigma)}\prod_{(r,t) \in \sigma} \langle \xi_r| \xi_t  \rangle_U = \chi(W(\xi_1) \cdot \ldots \cdot W(\xi_d)).\]

(2).  To exhibit the desired properties of $\Theta:= p \pi_\omega (\cdot) p$, we will verify conditions (i)--(iii) in Theorem \ref{thm:Nouembedding} for the $\ast$-homomorphism $\pi_\omega$. (i) follows immediately from part (1) of the present theorem.  For (ii), we note that by linearity and multiplicativity of $\pi_\omega$, it suffices to check condition (ii) on the generators $\pi_\omega(W(\xi)) = (u_m(\xi))_\omega$, $(\xi \in \HH_\C)$.  However, there is a minor issue here coming from the fact that for arbitrary $\xi \in \HH_\C$, there is no reason to expect elements $u_m(\xi) \in \Gamma_{q_0}(\R^m) \overline{\otimes} \Gamma_{q_1}(\HH \otimes \C^m)$ to even be analytic, let alone the sequence $(\sigma_{-i}(u_m(\xi)))_{m \in \N}$ be uniformly bounded.  To overcome this issue, put $\HH_\C^{an} = \bigcup_{\lambda > 1}{\textbf 1}_{[\lambda^{-1}, \lambda]}(A)\HH_\C$, where ${\textbf 1}_{[\lambda^{-1}, \lambda]}(A)$ denotes the spectral projection of the analytic generator $A$ corresponding to the interval $[\lambda^{-1},\lambda]$.  Following \cite[Theorem 3.1]{MR3599523}, we see that $\HH_\C^{an} \subset \HH_\C$ is a dense linear subspace such that 
$I \HH_\C^{an} = \HH_\C^{an}$.  Moreover, for each $\xi \in \HH_\C^{an}$, we have that $\xi$ (respectively $W(\xi)$) is analytic for the action of the unitary group $U_t = A^{it}$ (respectively the modular automorphism group $\sigma_t$), and  \[\sigma_zW(\xi) = W(A^{-iz}\xi) \qquad (z \in \C).\] 

In our present setting, we shall restrict the domain of $\pi_\omega$ to the $\ast$-subalgebra $\widetilde{\Gamma}_q(\HH)_{an} \subset \widetilde{\Gamma}_q(\HH)$, consisting of linear combinations of Wick words of the form $W(\xi)$ with $\xi \in (\HH_\C^{an})^{\odot n}$, $(n \in \N_0)$.  Since $\widetilde{\Gamma}_q(\HH)_{an}$ is still w$^\ast$-dense in $\Gamma_q(\HH)$ and is generated by $(W(\xi))_{\xi \in \HH_{\C}^{an}}$, we just have to show that the equivalence class representative $(u_m(\xi))_{m \in \N}$ for $\pi_\omega(W(\xi))$ satisfies condition (ii) of Theorem \ref{Th:Embedding} for each $\xi \in \HH_\C^{an}$.
To this end, note that on $\Gamma_{q_0}(\R^m) \overline{\otimes} \Gamma_{q_1}(\HH \otimes \C^m)$, we have  \[\sigma_t^m = \id_{\Gamma_{q_0}(\R^m)} \otimes \sigma_t^{\Gamma_{q_1}(\HH \otimes \C^m)} \quad \& \quad \sigma_t^{\Gamma_{q_1}(\HH \otimes \C^m)}(W(\xi \odot e)) = W(A^{-it}\xi \odot e) \qquad (\xi \in \HH_\C, \ e \in \C^m).\]  It follows from these identities that if $\xi \in \HH_\C^{an}$ and $e\in \C^m$, then elements $W(\xi \odot e)$ and $u_m(\xi)$ are analytic for their respective modular groups and
\begin{align*}
\sigma_z^m(u_m(\xi)) &= \Big(\frac{1}{\sqrt{m}} \sum_{k=1}^{m} W(e_k) \otimes \sigma_z^{\Gamma_{q_1}(\HH \otimes \C^m)}W(\xi \odot e_k)\Big) \\
&=\frac{1}{\sqrt{m}} \sum_{k=1}^{m} W(e_k) \otimes W(A^{-iz}\xi \odot e_k) = u_m(A^{-iz}\xi), \qquad (z \in \C).
\end{align*}
The uniform boundedness of the sequence $(\sigma_{-i}^mu_m(\xi))_{m \in \N}$ now follows along the same lines as that of $(u_m(\xi))_{m \in \N}$:
\[
\sup_m\|\sigma_{-i}^mu_m(\xi)\| = \sup_m\|u_m(A^{-1}\xi)\| \le 2(1-q_0)^{\frac{-1}{2}}\|W(A^{-1}\xi \odot e_1)\|.  
\]
For (iii), it again suffices by linearity and multiplicativity to verify that for all $\pi_\omega (W(\xi)) = (u_m(\xi))_\omega$, ($\xi \in \HH_\C$),
\[
p((\sigma_t^m(u_m(\xi)))_\omega)p \in p\mathsf B p,
\] 
where $\mathsf B$ is the w$^\ast$-closure of $\pi_\omega(\widetilde{\Gamma}_q(\HH))$ in $\mathsf A$.  But this last point is obvious, because by the previous computation, $\sigma_t^m(u_m(\xi)) =  u_m(A^{-it}\xi)$ for all $m$, giving
\[
p((\sigma_t^m(u_m(\xi)))_\omega)p  = p((u_m(A^{-it}\xi))_\omega) p = p \pi_\omega(W(A^{-it}\xi))p  \in  p\mathsf B p.
\]
\end{proof}

\section{Transferring radial multipliers} \label{transfer}
The main aim of this section is to use the ultraproduct embedding result (Theorem \ref{Th:Embedding}) of the previous section to establish the following transference result for radial multipliers on $q$-Araki-Woods algebras.  In what follows, we freely use the notation of the previous sections.

\begin{thm} \label{Th:Transference}
Let $\varphi: \N \to \C$ be a function such that the associated radial multipliers $\mathsf{m}_{\varphi}: \Gamma_q(\R^m) \to \Gamma_q(\R^m)$  have completely bounded norms uniformly bounded in $m$. Then the radial multiplier defined by $\varphi$ on any $q$-Araki-Woods algebra $\Gamma_q(\HH)$ is completely bounded and 
\[
\|\mathsf{m}_{\varphi}: \Gamma_q(\mathsf{H}) \to \Gamma_q(\mathsf{H})\|_{\op{cb}} \le  \sup_{m \in \N} \|\mathsf{m}_{\varphi}: \Gamma_q(\R^m) \to \Gamma_q(\R^m) \|_{\op{cb}} = \|\mathsf{m}_{\varphi}:\Gamma_q(\ell_{2,\R})\to\Gamma_q(\ell_{2,\R})\|_{\op{cb}}. 
\]
\end{thm}

The main technical tool in establishing  Theorem \ref{Th:Transference} is the following intertwining-type property for projections onto Wick words of a given length with respect to the ultraproduct embedding given by Theorem \ref{Th:Embedding}

\begin{thm}\label{Th:Intertwining}
Let $\Gamma_q(\mathsf{H})$ be a $q$-Araki-Woods algebra. Let $P_n\from \Gamma_q(\mathsf{H})\to \Gamma_q(\mathsf{H})$ be the projection onto the ultraweakly closed span of $\{W(\xi): \xi \in \mathsf{H}_{\C}^{\odot n}\}$. Then, using the notation from Theorem \ref{Th:Embedding}, we have 
\begin{align} \label{int-eqn}
\Theta \circ P_n = p(P_n \otimes \op{Id})_\omega p \circ \Theta.
\end{align}
\end{thm}
There are two things that have to be verified in Theorem \ref{Th:Intertwining}. The first one, which is a routine check, is to prove that $(P_n \otimes \op{Id})_\omega$ (and therefore also the composition $p(P_n \otimes \op{Id})_\omega p$) is a well-defined map on the (Raynaud) ultraproduct $\mathsf A$. Using Theorem \ref{Th:Polynomial bound}, we can show that $(P_n \otimes \op{Id})_\omega$ is well defined on the $C^{\ast}$-ultraproduct $\widetilde{\mathsf A} \subset \mathsf A$. To conclude, we have to verify that it extends to a normal map on $\mathsf A$. Since we are dealing with the Raynaud ultraproduct, the predual of our ultraproduct is equal to the Banach space ultraproduct of preduals. On each level we can take the predual map of $(P_n \otimes \op{Id})_{m \in \N}$ and use this sequence to obtain a map $\Psi$ on the ultraproduct of $L^{1}$-spaces, the predual of the ultraproduct. The dual of $\Psi$ coincides with $(P_n \otimes \op{Id})_\omega$ on the $C^{\ast}$-ultraproduct, hence it is its unique normal extension. A similar argument is presented, for instance, in \cite[Lemma 3.4]{MR2822210}.

The second step in proving Theorem \ref{Th:Intertwining} is to understand the images of Wick words under the $\ast$-homomorphism $\pi_{\omega}:\widetilde{\Gamma}_q(\HH) \to \mathsf A$. To accomplish this, for any $d \in \N$ and $\xi_1,\dots,\xi_d \in \mathsf{H}_{\C}$, we define elements $W^{s}(\xi_1\odot \dots \odot \xi_d) \in \mathsf A$ by setting
\[
W^{s}(\xi_1\odot \dots \odot \xi_d) := \left(m^{-\frac{d}{2}} \sum_{\substack{k\from [d] \to [m] \\ \textrm{injective}}} W(e_{k(1)})\dots W(e_{k(d)}) \otimes W(\xi_1 \odot e_{k(1)})\dots W(\xi_d \odot e_{k(d)})\right)_\omega.
\]
Because we are summing over distinct indices, the vectors $e_{k(1)},\dots, e_{k(d)}$ are pairwise orthogonal, so $W(e_{k(1)})\dots W(e_{k(d)}) = W(e_{k(1)} \odot \dots \odot e_{k(d)})$. One can then use the Khintchine inequality (Corollary \ref{cor:NCKhintchine}) to prove that the sequence defining $W^{s}(\xi_1\odot \dots \odot \xi_d)$ is uniformly bounded, hence defines a legitimate element of the ultraproduct. We will not give more details here because in the next proposition we show that $\pi_{\omega}(W(\xi_1\odot \dots \odot \xi_d)) = W^{s}(\xi_1\odot \dots \odot \xi_d)$, so it definitely is an element of the ultraproduct.
\begin{thm}\label{Th:Wickultraproduct}
Let $\xi_1,\dots,\xi_d \in \mathsf{H}_{\C}$. Let $\pi_{\omega}$ be as in Theorem \ref{Th:Embedding}. Then $\pi_{\omega}(W(\xi_1\odot \dots \odot \xi_d)) = W^{s}(\xi_1\odot \dots \odot \xi_d)$.
\end{thm}
\begin{proof}
We proceed by induction on $d \in \N_0$. The base cases $d=0,1$ are obvious from the definitions.  Now assume that the claimed formula is true for all lengths $0 \le d' \le d$, and consider the $d+1$ case.  Fix $\xi_0, \xi_1, \ldots, \xi_d \in \mathsf{H}_{\C}$.  It then follows from Proposition \ref{prop:Wick} that the following relation
holds.
\[
W(\xi_0\odot \dots \odot \xi_d) = W(\xi_0)W(\xi_1\odot \dots \odot \xi_d)- \sum_{l=1}^d q^{l-1}\braket{I\xi_0}{\xi_l}_U W(\xi_1 \odot \ldots \odot \widehat{\xi_l} \odot \ldots \odot \xi_d),
\]
where, as usual, $\widehat{\xi_l}$ means that the tensor factor $\xi_l$ is deleted from the simple tensor under consideration.
Applying $\pi_{\omega}$ to this relation and using our induction hypothesis, we have
\begin{equation}\label{Form:Wick}
\pi_{\omega}(W(\xi_0\odot \ldots \odot\xi_d)) = W^{s}(\xi_0)W^{s}(\xi_1 \odot \ldots\odot \xi_d) - \sum_{l=1}^d q^{l-1}\braket{I\xi_0}{\xi_l}_U W^s(\xi_1\odot \ldots \odot \widehat{\xi_l} \odot \ldots \odot \xi_d).
\end{equation}
Next, we expand the first term on the right-hand side in the above equation:
\begin{align*}
&W^s(\xi_0)W^s(\xi_1\odot \dots \odot \xi_d)\\
&=\left(m^{-\frac{1}{2}}\sum_{k(0)=1}^m W(e_{k(0)}) \otimes W(\xi_0 \odot e_{k(0)})\right)_\omega \\
&\times  \left(m^{-\frac{d}{2}}\sum_{\substack{k\from [d]\to [m] \\ \textrm{injective}}}W(e_{k(1)})\dots W(e_{k(d)}) \otimes W(\xi_1 \odot e_{k(1)})\dots W(\xi_d \odot e_{k(d)})\right)_\omega\\
&= \left(m^{-\frac{d+1}{2}}\sum_{\substack{k\from [d]_{0}\to [m] \\ \textrm{injective}}}W(e_{k(0)})\dots W(e_{k(d)}) \otimes W(\xi_0 \odot e_{k(0)})\ldots W(\xi_d \odot e_{k(d)})\right)_\omega \\
&+ \left(m^{-\frac{d+1}{2}}\sum_{k(0)=1}^m \sum_{l=1}^d \sum_{\substack{k\from[d] \to [m] \\ \textrm{injective} \\ k(l)=k(0)}}W(e_{k(0)})\dots W(e_{k(d)}) \otimes W(\xi_0 \odot e_{k(0)})\ldots W(\xi_d \odot e_{k(d)})\right)_\omega \\
&=W^s(\xi_0 \odot \xi_1 \odot \dots \odot \xi_d) \qquad (\text{this is the first term in the preceding sum}) \\
&+ \sum_{l=1}^d\left(m^{-\frac{d+1}{2}}\sum_{k(0)=1}^m \sum_{\substack{k\from[d]\to[m]\\ \textrm{injective} \\  k(l)=k(0)}}W(e_{k(0)})\ldots W(e_{k(d)}) \otimes W(\xi_0 \odot e_{k(0)})\dots W(\xi_d \odot e_{k(d)})\right)_\omega. 
\end{align*}
The first term is already a part of what we wanted, but we also have to deal with the second term. Note that for $k(0)=k(l)$ and $k(1)\neq \dots \neq k(d)$ we have 
\begin{equation}\label{Form:Wickmult1}
W(e_{k(0)})\dots W(e_{k(d)}) = W(e_{k(0)} \odot \dots \odot e_{k(d)}) + q_0^{l-1} W(e_{k(0)})\dots \widehat{W(e_{k(l)})} \dots W(e_{k(d)})
\end{equation}
and
\begin{align}\label{Form:Wickmult2}
&W(\xi_0 \odot e_{k(0)}) \dots  W(\xi_d \odot e_{k(d)})) \\
&=W((\xi_0\odot e_{k(0)})\odot\dots \odot (\xi_d \odot e_{k(d)}))\notag \\
&+ \braket{I\xi_0 \odot e_{k(0)}}{\xi_l\odot e_{k(l)}}_U q_1^{l-1} W(\xi_0 \odot e_{k(0)})\dots\widehat{W(\xi_l \odot e_{k(l)})}\dots W(\xi_d \odot e_{k(d)}) \notag\\
&= W((\xi_0\odot e_{k(0)})\odot\dots \odot (\xi_d \odot e_{k(d)})) \notag\\
&+ \braket{I\xi_0}{\xi_l}_U q_1^{l-1} W(\xi_0 \odot e_{k(0)})\dots\widehat{W(\xi_l \odot e_{k(l)})}\dots W(\xi_d \odot e_{k(d)}).\notag
\end{align}
Indeed, if $(v_1,\dots,v_n) \subseteq \HH_\C$ is a family of orthogonal vectors then $W(v_1)\dots W(v_n) = W(v_1\odot\dots\odot v_n)$, as we remarked earlier. In our case we have a sequence $(w,v_1,\dots, v_d)$, where $Iw$ is orthogonal to all vectors $v_j$ for $j\neq l$, so we get 
\begin{align*}
W(w)W(v_1)\dots W(v_d)\Omega &= (a^{\ast}(w) + a(Iw)) v_1\odot \dots \odot v_d \\
&= w \odot v_1\odot \dots \odot v_d + a(Iw) (v_1 \odot  \dots \odot v_d) \\
&= w \odot v_1 \odot \dots \odot v_d + q_1^{l-1} \braket{Iw}{v_l}_U v_1 \odot \dots \odot \widehat{v_l} \odot  \dots \odot v_d, 
\end{align*}
hence the formula above. Tensoring $W(e_{k(0)})\dots W(e_{k(d)})$ with $ W(\xi_0 \odot e_{k(0)}) \dots  W(\xi_d \odot e_{k(d)})$ (keeping in mind that $q_0q_1=q$) gives us four terms, one of which is
\[
q^{l-1}\braket{I\xi_0}{\xi_l}_U W(e_{k(1)})\dots\widehat{W(e_{k(l)})}\dots W(e_{k(d)})\otimes W(\xi_0 \odot e_{k(0)})\dots\widehat{W(\xi_l \odot e_{k(l)})}\dots W(\xi_d \odot e_{k(d)})
\]
and we will deal with the three other terms later. To these expressions we need to apply the sum $\sum_{l=1}^{d}m^{-\frac{d+1}{2}}\sum_{k(0)=1}^m \sum_{\substack{k\from [d]\to[m]\\ \textrm{injective} \\ k(l)=k(0)}}$. Since $k(l)$ and $k(0)$ are omitted, we can forget about the condition $k(l)=k(0)$ and perform the sum over $k(0)$ immediately, resulting in a sum $\sum_{l=1}^{d}m^{-\frac{d-1}{2}}\sum_{\substack{k\from [d]\setminus\{l\}\to [m] }}$. Without the sum over $l$, this is the sum over $d-1$ distinct indices appearing in the definition of $W^s$, so we get the sum
\[
\sum_{l=1}^{d}q^{l-1}\braket{I\xi_0}{\xi_l}_U W^{s}(\xi_0 \odot \dots \odot \widehat{\xi_l} \odot  \dots \odot  \xi_d). 
\]
To sum up, we have checked so far that
\[
W^{s}(\xi_{0})W^{s}(\xi_1\odot \dots \odot \xi_d) = W^s(\xi_0\odot \dots\odot \xi_d) + \sum_{l=1}^{d} q^{l-1} \braket{I\xi_0}{\xi_l}_UW^{s}(\xi_0 \odot \dots \odot \widehat{\xi_l} \odot \dots \odot \xi_d) + R,
\]
where $R$ is the ``remainder'' term that will turn out to be a zero element of the ultraproduct. Inserting this into \eqref{Form:Wick} we get that
\[
\pi_{\omega}(W(\xi_1\odot \dots \odot \xi_d)) = W^{s}(\xi_1\odot \dots \odot \xi_d) + R,
\]
so if we can check that $R$ is really a zero element then this ends the proof. 
\end{proof}
Let us just recall that $R$ comes from the three neglected so far terms arising from tensoring $W(e_{k(0)})\dots W(e_{k(d)})$ with $ W(\xi_0 \odot e_{k(0)}) \dots  W(\xi_d \odot e_{k(d)})$. It can be written as 
\[R = \left(m^{-\frac{d+1}{2}}\sum_{l=1}^{d} (R_{1,l}(m) + \braket{I\xi_0}{ \xi_l} q_1^{l-1}R_{2,l}(m) + q_{0}^{l-1}R_{3,l}(m))\right)_\omega,
\]
where:
\[
R_{1,l}(m)= \sum_{k(0)=1}^m \sum_{\substack{k\from[d]\to [m]\\\textrm{injective} \\ k(l)=k(0)}}W(e_{k(0)} \odot \ldots \odot e_{k(d)})\otimes W((\xi_0 \odot e_{k(0)})\odot \ldots\odot (\xi_d \odot e_{k(d)})),
\]
\[
R_{2,l}(m) = \sum_{k(0)=1}^m  \sum_{\substack{k\from[d] \to [m]\\ \textrm{injective} \\ k(l)=k(0)}}W(e_{k(0)} \odot \ldots \odot e_{k(d)}) \otimes  W(\xi_1 \odot e_{k(1)})\odot \widehat{W(\xi_l \odot e_{k(l)})} \odot W(\xi_d \odot e_{k(d)}),
\]
and
\[
R_{3,l}(m) = \sum_{k(0)=1}^m \sum_{\substack{k\from [d] \to [m]\\ \textrm{injective} \\ k(l)=k(0)}}W(e_{k(1)})\ldots\widehat{W(e_{k(l)})} \ldots W(e_{k(d)}) \otimes  W((\xi_0 \odot e_{k(0)})\odot \ldots\odot (\xi_d \odot e_{k(d)})).
\]
Recall the formulas \eqref{Form:Wickmult1} and \eqref{Form:Wickmult2}. After tensoring the right-hand sides we get four terms, one of which was already incorporated in the proof of Theorem \ref{Th:Wickultraproduct}. The other three are:
\[
W(e_{k(0)} \odot \ldots \odot e_{k(d)}) \otimes W((\xi_0 \odot e_{k(0)})\odot \ldots\odot (\xi_d \odot e_{k(d)})),
\]
\[
q_1^{l-1} \braket{I\xi_0}{\xi_l}W(e_{k(0)} \odot \ldots\odot e_{k(d)}) \otimes  W(\xi_1 \odot e_{k(1)})\ldots \widehat{W(\xi_l \odot e_{k(l)})} \ldots W(\xi_d \odot e_{k(d)}),
\]
and
\[
q_0^{l-1} W(e_{k(1)})\ldots\widehat{W(e_{k(l)})} \ldots W(e_{k(d)}) \otimes  W((\xi_0 \odot e_{k(0)})\odot \ldots\odot (\xi_d \odot e_{k(d)})).
\]
To obtain $R$, we just need to take sums over appropriate sets of indices.

We will now examine properties of $R$. Since $q_0,q_1$, and  the range of summation over $l$ is fixed, to show that $R$ is a zero element in the ultraproduct, it suffices to show that $\lim_{m\to\infty} m^{-\frac{d+1}{2}}\|R_{i,l}\|=0$ for any $1 \le  i \le 3$ and $l \in [d]$. We will use Nou's noncommutative Khintchine inequality for this (Corollary \ref{cor:NCKhintchine}), but before that we need to obtain a bound for the coefficients.
\begin{lem}\label{Lem:coefficients}
There exists a constant $D(d) > 0$ (depending only on the initial choice of $\xi_1, \ldots, \xi_d \in \HH_\C$) such that for all $m \in \N$ and all $k:[d] \to [m]$, the following inequalities hold:
\begin{align*}
\|W((\xi_{0} \odot e_{k(0)})\odot \dots\odot (\xi_d \odot e_{k(d)}))\| &\leqslant D(d) \\
 \|W(\xi_1 \odot e_{k(1)})\dots\widehat{W(\xi_l \odot e_{k(l)})} \ldots W(\xi_d \odot e_{k(d)})\| &\leqslant D(d) \\
\|W(e_{k(1)})\dots\widehat{W(e_{k(l)})} \dots W(e_{k(d)})\| &\leqslant D(d).
\end{align*}

\end{lem}

\begin{proof}
The second and third inequality will follow if we can show that there is a constant $D>0$ such that $\|W(\xi_r \odot  e_{k(r)})\|,  \|W(e_{k(r)})\| \leqslant D$ (independently of $r \in [d]$). But the existence of $D$ follows from the simple fact for any $q$-Araki-Woods algebra $\Gamma_q(\HH)$ and $\xi \in \HH_\C$, we have  $\|W(\xi)\|_{\Gamma_q(\HH)} \leqslant \|a_q^*(\xi)\| +\|a_q(I\xi)\| \le 2(1 - |q|)^{-1/2}\max\{\|\xi\|, \|I\xi\|\}$.  
Now consider the first inequality. By the Khintchine inequality with $\mathsf{K}=\C$ (Corollary \ref{cor:NCKhintchine}), the left-hand side is bounded by \[C(q_1)(d+1)\max_{0 \le l \le d} \|(\mathds{1}_{d-l} \odot \mathcal{I})(R_{d,l}^{\ast}((\xi_{0} \odot e_{k(0)})\odot \dots\odot (\xi_d \odot e_{k(d)})))\|.\] Writing the above $(\mathds{1}_{d-l} \odot \mathcal{I})R_{d,k}^{\ast}$ terms as sums of simple tensors, one easily sees that the corresponding norms are bounded by a constant depending only on $d$.  (Note that the unboundedness of $\mc I$ plays no role here, as $\xi_0, \ldots, \xi_d \in \HH_\C$ remain fixed.)
\end{proof}
We need one more proposition.   In the following, $m \in \N$ and $\xi_0, \ldots, \xi_d$ are fixed as usual.   Let $I_{l}$ denote the set of indices $(k(0),\dots, k(d)) \in [m]^{d+1}$ that are pairwise distinct except for the pair $(k(0),k(l))$; a generic element of $I_{l}$ will be called $\textbf{i}$ and the corresponding tensor $e_{k(0)}\otimes \dots\otimes e_{k(d)}$ will also be denoted by $\textbf{i}$. We will denote $W(e_{k(0)}\otimes \dots \otimes e_{k(d)})$ by $W_{\textbf{i}}$ and $W((\xi_{0} \odot e_{k(0)})\otimes\dots\otimes(\xi_{d} \odot e_{k(d)}))$ by $W_{\textbf{i}}^{\xi}$.
\begin{prop}\label{Prop:remainderbound}
Given any Hilbert space $\mathsf K$ and any family of operators $(A_{\textbf{i}})_{\textbf{i} \in I_l} \subset \op B(\mathsf K)$, the following inequalities hold.
\begin{align*}
\|\sum_{I_l} A_{\textbf{i}} \otimes W_{\textbf{i}}\| &\leqslant C(d) \sup_{\textbf{i} \in I_{l}} \|A_{\textbf{i}}\| m^{\frac{d}{2}}  \\
\|\sum_{I_l} A_{\textbf{i}} \otimes W_{\textbf{i}}^{\xi}\| &\leqslant C(d) \sup_{\textbf{i} \in I_{l}} \|A_{\textbf{i}}\| m^{\frac{d}{2}},
\end{align*}
where $C(d) > 0$ depends only on $d$ and the choice of vectors $\xi_0,\xi_1,\dots,\xi_d \in \HH_\C$. 
\end{prop}
\begin{proof}
The proofs of both inequalities are essentially the same. We will deal with the first one; to obtain a proof of the second one has to apply conjugation in some places but since we are dealing with a fixed number of vectors $\xi_0,\dots,\xi_d$, the unboundedness of conjugation does not play any role. By the Khintchine inequality (Corollary \ref{cor:NCKhintchine}) we need to deal with
\[
\max_{0\leqslant k \leqslant d+1}\|\sum_{I_{l}} A_{\textbf{i}} \otimes R_{d+1,k}^{\ast} (\textbf{i})\|,
\]
up to a $d$-dependent constant.
 
Since $R_{d+1,k}^{\ast}$ is a sum of operators that only permute vectors, and the coefficients of this sum are summable, we just need to take care of a single term of the form
\[
\max_{0\leqslant k \leqslant d+1}\|\sum_{I_{l}} A_{\textbf{i}} \otimes \sigma(\textbf{i})_{(d+1,k)}\|,
\]
where $\sigma$ denotes the action of the permutation and the decoration $(d+1,k)$ reminds us of the fact that $\sigma(\textbf{i})$ is viewed now as an element of $\HH_{c}^{\otimes (d+1-k)} \otimes_{h} \HH_{r}^{\otimes k}$. Whatever the $\sigma$, the tensor $\sigma(\textbf{i})_{d+1,k}$ is always of the form $e_{i_{0}} \odot\dots\odot e_{i_{d-k}} \otimes e_{i_{d-k+1}} \odot \dots \odot e_{i_d}$, where for different indices $\textbf{i}$ and $\textbf{i}'$ these tensors are different. The key property that we will need is that we have two orthonormal systems $(v_s)_{s \in S} \subset \mathsf{H}^{\otimes (d+1-k)}$ and $(w_j)_{j \in J} \subset \mathsf{H}^{\otimes k} $ such that for any $\textbf{i} \in I_l$ we have $\sigma(\textbf{i})_{d+1,k} = v_s \otimes w_j$ for some $s\in S$ and $j\in J$. Therefore we can get rid of the sign $\sigma$ and just consider
\[
\max_{0\leqslant k \leqslant d+1}\|\sum_{I_{l}} A_{\textbf{i}} \otimes \textbf{i}_{(d+1,k)}\|,
\]
Since we are dealing with tensor powers of $\mathsf{H}$ equipped with $q$-deformed inner products, we would rather have families $(v_s')_{s\in S}$ and $(w_j')_{j\in J}$ that are orthonormal in $\mathsf{H}_q^{\otimes (d+1-k)} $ and $\mathsf{H}_q^{\otimes k}$, respectively. To achieve this, we will use the operators defining the $q$-deformed inner products, $P_q^{d+1-k}$ and $P_q^k$. Let $\xi(\textbf{i})_{d+1,k}$ be tensors defined by $((P_q^{d+1-k})^{\frac{1}{2}} \otimes (P_q^k)^{\frac{1}{2}})(\xi(\textbf{i})_{d+1,k}) =\textbf{i}_{d+1,k}$. Then we can write $\xi(\textbf{i})_{d+1,k} = v_s' \otimes w_j'$ for some tensors $v_s'$ and $w_j'$ coming from orthonormal families in $\mathsf{H}_q^{\otimes (d+1-k)}$ and $\mathsf{H}_q^{\otimes k}$.  Since the row/column Hilbert spaces are homogeneous operator spaces (and Haagerup tensor product allows tensoring cb maps) we can bound $\max_{0\leqslant k \leqslant d+1}\|\sum_{I_{l}} A_{\textbf{i}} \otimes \textbf{i}_{d+1,k}\|$ by $\max_{0\leqslant k \leqslant d+1}\|\sum_{I_{l}} A_{\textbf{i}} \otimes \xi(\textbf{i})_{d+1,k}\| $, up to a $d$-dependent constant coming from the norms of $(P_q^{d+1-k})^{\frac{1}{2}}$ and $(P_q^k)^{\frac{1}{2}}$. Because we are using the Haagerup tensor product, we have the following completely isometric isomorphism $\mathsf{H}_{c} \otimes_{h} \overline{\mathsf{K}}_{r} \simeq \mathcal{K}(\mathsf{K}, \mathsf{H})$. Under this identification the tensors $\xi(\textbf{i})_{d+1,k}$ correspond to matrix units in $\mathcal{K}(\mathsf{H}_q^{\otimes k}, \mathsf{H}_q^{\otimes (d+1-k)})$. This means that the operators $A_{\textbf{i}}$ fill different entries in a large operator matrix. By comparing the operator norm with the Hilbert-Schmidt norm we get the estimate
\[
\left\|\sum_{\textbf{i}\in I_l} A_{\textbf{i}} \otimes W_{\textbf{i}} \right\| \leqslant C(d) \left(\sum_{\textbf{i} \in I_l} \|A_{\textbf{i}}\|^2\right)^{\frac{1}{2}} \le C(d) \left( |I_l| \sup_{\textbf{i}\in I_l} \|A_{\textbf{i}}\|^2\right)^\frac{1}{2}, 
\]
which can be further bounded by
\[
C(d) \left( m^d \sup_{\textbf{i}\in I_l} \|A_{\textbf{i}}\|^2\right)^\frac{1}{2} = C(d) m^{\frac{d}{2}} \sup_{\textbf{i} \in I_l} \|A_{\textbf{i}}\|.
\]
\end{proof}
Finally, to conclude that $R=0$ in the ultraproduct, we just observe that each  component $R_{i,l} = (R_{i,l}(m))_{m \in \N}$ is a a sequence of terms of the form appearing in Proposition \ref{Prop:remainderbound} with coefficients $(A_{\textbf{i}}(m))_{m \in \N, \textbf{i} \in I_l}$ uniformly bounded in $\textbf{i}$ and $m$ by the constant $D(d)$ from Lemma \ref{Lem:coefficients}, so the norm $m^{-\frac{d+1}{2}} R_{i,l}(m)$ is bounded from above by $C(d)D(d) m^{-\frac{1}{2}}$, and hence tends to zero. This finishes the proof of Theorem \ref{Th:Wickultraproduct}. With this tool at hand, we prove Theorem \ref{Th:Intertwining}.
\begin{proof}[Proof of Theorem \ref{Th:Intertwining}]
Let $W(\xi)$ be a Wick word associated with $\xi \in \mathsf{H}_{\C}^{\odot d}$. Then we easily obtain $\pi_{\omega} (P_n W(\xi)) = \delta_{n,d} W^{s} (\xi)$. On the other hand, let us first apply $\pi_{\omega}$ to obtain $W^{s}(\xi)$. Since, as we already remarked earlier, $W(e_{k(1)})\dots W(e_{k(d)})=W(e_{k(1)}\otimes\dots\otimes e_{k(d)})$, the operators acted on by the $P_n$ part of the operator $(P_n \otimes \op{Id})_\omega$ are exactly of length $n$. Therefore $(P_n \otimes \op{Id})_\omega W^{s}(\xi) = \delta_{n,d} W^{s}(\xi)$.  By linearity, this implies that 
$\pi_\omega  \circ P_n  = (P_n \otimes \op{Id})_\omega \circ \pi_\omega$ on the algebra of Wick words $\widetilde{\Gamma}_q(\HH)$.  Compressing by the support projection $p$, we then obtain \[\Theta \circ P_n = p  (P_n \otimes \op{Id})_\omega \circ \pi_\omega (\cdot)  p = p \ (P_n \otimes \op{Id})_\omega p \circ \Theta \quad \text{on} \quad \widetilde{\Gamma}_q(\HH),\]
where in the second equality we used the fact that $p \in \pi_\omega(\widetilde{\Gamma}_q(\HH))'$ (see \cite[Lemma 4.1]{MR2200739}). Since the desired equality holds on the ultraweakly dense subset $\widetilde{\Gamma}_q(\HH)$, and all maps under consideration are normal, equality holds everywhere. 
\end{proof}

Let us now furnish a proof of the transference result for radial multipliers. 

\begin{proof}[Proof of Theorem \ref{Th:Transference}]
From Theorem \ref{Th:Intertwining} we get that $\Phi\circ\mathsf{m}_{\varphi}( x) = p(\mathsf{m}_{\varphi} \otimes \op{Id})_\omega p \circ \Phi( x)$ for any $x = W(\xi)$ with $\xi \in (\HH_\C)^{ \odot d}$. By linearity we can extend this equality to all $x \in \widetilde{\Gamma}_q(\HH)$. It follows that we have control on the cb norm of $\mathsf m_\varphi$ acting on the norm-closure of finite Wick words, i.e. on the $C^{\ast}$-algebra $\mc A_q(\HH)$. Since $\mathsf m_\varphi$ is automatically normal (cf. \cite[Lemma 3.4]{MR2822210}), it extends to a normal map on $\Gamma_q(\HH)$ with the same cb norm, so we get
\[
\|\mathsf{m}_{\varphi}:\Gamma_q(\mathsf{H}) \to \Gamma_q(\mathsf{H})\|_{\op{cb}} \le \sup_{m \in \N} \|\mathsf{m}_{\varphi}: \Gamma_q(\R^m) \to\Gamma_q(\R^m)\|_{\op{cb}}.
\]
Since $\Gamma_q(\R^m)$ is a subalgebra of $\Gamma_q(\R^{m+1})$ which is the range of a normal faithful trace-preserving conditional expectation that intertwins the action of $\mathsf{m}_{\varphi}$, the sequence of norms on the right-hand side is non-decreasing, so
\[
\|\mathsf{m}_{\varphi}: \Gamma_q(\mathsf{H}) \to \Gamma_q(\mathsf{H})\|_{\op{cb}} \le\lim_{m\to\infty} \|\mathsf{m}_{\varphi}: \Gamma_q(\R^m) \to \Gamma_q(\R^m)\|_{\op{cb}}.
\]
By the same token, this limit is not greater than $\|\mathsf{m}_{\varphi}: \Gamma_q(\ell_{2,\R}) \to \Gamma_q(\ell_{2,\R})\|_{\op{cb}}$. Since the union of the algebras $\Gamma_q(\R^m)$ is strongly dense in $\Gamma_q(\ell_{2,\R})$, the union of the preduals is norm-dense in the predual of $\Gamma_q(\ell_{2,\R})$. Therefore the limit of norms is equal to the norm of the multplier defined on $L^{1}(\Gamma_q(\ell_{2,\R}))$. By dualising, we get that 
\[
\lim_{m \to \infty} \|\mathsf{m}_{\varphi}: \Gamma_q(\R^m) \to \Gamma_q(\R^m)\|_{\op{cb}} = \|\mathsf{m}_{\varphi}: \Gamma_q(\ell_{2,\R}) \to \Gamma_q(\ell_{2,\R})\|_{\op{cb}}.
\]
\end{proof}

Let us conclude this section with an application to the extension of Theorem \ref{Th:Polynomial bound} to general $q$-Araki-Woods algebras.

\begin{cor}
Let $\Gamma_q(\mathsf{H})$ be a $q$-Araki-Woods algebra. Let $P_n$ be the projection onto Wick words of length $n$, defined by $P_n W(\xi) = \delta_{n,d} W(\xi)$, where $\xi \in \mathsf{H}_{\C}^{\odot d}$. Then $P_n$ extends to a completely bounded, normal map on $\Gamma_{q}(\mathsf{H})$ and $\|P_n\|_{cb} \leqslant C(q)^2 (n+1)^2$.
\end{cor}
\begin{proof}
We just observe that $P_n = \mathsf m_{\varphi_n},$ where $\varphi_n$ is the Kroenecker delta-function $\varphi_n(k) = \delta_{k,n}$.  By Theorems \ref{Th:Polynomial bound} and \ref{Th:Transference}, we obtain $\|P_n:\Gamma_q(\HH) \to \Gamma_q(\HH)\|_{\op{cb}} \leqslant\|P_n:\Gamma_q(\ell_{2,\R}) \to \Gamma_q(\ell_{2,\R})\|_{\op{cb}}\leqslant C(q)^2 (n+1)^2$.
\end{proof}

The last section will be devoted to the proof of the complete metric approximation property for $\Gamma_q(\HH)$.
\section{Proof of Theorem \ref{Th:CMAP} } \label{proofs}
Before proving our main result, we need to recall one more lemma.
\begin{lem}[{\cite[Proposition 3.17]{MR2822210}}]
Let $\mathsf{H}$ be the Hilbert space constructed from the pair $(\mathsf{H}_{\R}, (U_t)_{t \in \R})$. Let $I$ be the complex conjugation on $\mathsf{H}_{\C}$. Then there exists a net $(T_{i})_{i\in I}$ of finite-rank contractions on $\mathsf{H}$ that satisfy $I T_i I = T_i$, i.e. preserve $\mathsf{H}_{\R}$, and converge strongly to identity.
\end{lem}
\begin{proof}[Proof of Theorem \ref{Th:CMAP}]
We define a net $\Gamma_{n,t,i} := \Gamma_q(e^{-t} T_i) Q_n$, where $n\in\N$, $t>0$, $i\in I$, the finite-rank maps $T_i$ come from the previous lemma, and $Q_n= P_0+\dots+P_n = \mathsf m_{\chi_{\{0,1, \ldots, n\}}}$ is the radial multiplier which projects onto Wick words of length at most $n$. Each $ \Gamma_{n,t,i}$ is a finite rank map on $\Gamma_q(\mathsf{H})$; indeed, $Q_n$ tells us that we have only Wick words of bounded length and $T_i$ tells us that we can only draw vectors from a finite dimensional Hilbert space, so we are left with a space of the form $\oplus_{d=0}^{n} \left(\C^{m}\right)^{\otimes d}$, which is finite-dimensional. We will pass to a limit with $i \to \infty$, $n\to \infty$ and $t\to 0$. The rate of convergences of $t$ and $n$ will not be independent and will be chosen in a way that assures the convergence $\|\Gamma_{n,t,i}\|_{cb} \to 1$. Let us check now that it is possible, using a standard argument of Haagerup (note that $\Gamma_q(e^{-t})P_k = e^{-kt} P_k$):
\begin{align*}
\|\Gamma_{n,t,i}\|_{cb}&= \|\Gamma_q(e^{-t} T_{i}) Q_n\|_{cb} \\
&\leqslant \|\Gamma_q(e^{-t}) Q_n\|_{cb} \\
&\leqslant \|\Gamma_q(e^{-t})\|_{cb} + \|\Gamma_q(e^{-t})(\mathds{1} - Q_n)\|_{cb} \\
&\leqslant 1 + \sum_{k>n} e^{-kt} \|P_k\|_{cb} \\
&\leqslant 1 + C(q)^2\sum_{k>n} e^{-kt} (k+1)^2.
\end{align*}
Since the series $\sum_{k\ge 0} e^{-kt} (k+1)^2 $is convergent, for any $t>0$ the sum will tend to zero when $n\to\infty$. Therefore we can choose the parameters $i, n \to \infty$ and $t\to 0$ such that the completely bounded norms of the operators $\Gamma_{n,t,i}$ tend to $1$. Then the operators $\frac{\Gamma_{n,t,i}}{\|\Gamma_{n,t,i}\|_{cb}}$ are completely contractive. We have to check that they converge ultraweakly to $\mathds{1}$. Since the denominators converge to $1$ and the net is uniformly bounded, it suffices to prove strong convergence on a linearly dense set. It is very easy to verify that the convergence holds for finite simple tensors, so this ends the proof. 
\end{proof}
Let us state two corollaries of (the proof) of this theorem.
\begin{cor}
Let $\mathsf{H}$ be the Hilbert space constructed from the pair $(\mathsf{H}_{R}, (U_t)_{t\in \R})$. Consider the $\sigma$-weakly dense $C^{\ast}$-algebra $\mc A_q(\mathsf{H}) \subseteq \Gamma_q(\HH)$ generated by the set $\{ W(\xi): \xi \in \mathsf{H}_{\R}\} \subset \op{B}(\mathcal{F}_q(\mathsf{H}))$. This $C^{\ast}$-algebra has the complete metric approximation property.
\end{cor}
\begin{proof}
Consider once again the maps $\Gamma_{n,t,i}:= \Gamma_q(e^{-t} T_i) Q_n$. The ranges of these maps are contained in $\widetilde{\Gamma}_q(\mathsf{H})$, the bounds for the norms remain the same, so it suffices to check the pointwise convergence in norm. Since the maps are uniformly bounded, it suffices to check the convergence on a linearly dense set, hence we may assume that $x= W(\xi_1\odot\dots\odot \xi_k)$. If $n$ is large enough the $Q_n$ that appears in the definition of $\Gamma_{n,t,i}$ has no effect on $x$, so we get
\[
\Gamma_{n,t,i} x - x = e^{-kt}W(T_{i}\xi_1\odot\dots\odot T_i \xi_k) - W(\xi_1\odot\dots\odot \xi_k).
\]
This last expression is easily seen to converge to zero in norm as $t \to 0$ and $i\to \infty$.  This can be seen either using the  Khintchine inequality (Corollary \ref{cor:NCKhintchine}), or just by expressing $W(\xi_1\odot\dots\odot \xi_k)$ as a non-commutative polynomial in $a_q(\xi_r)$'s and $a_q^*(\xi_r)$'s and invoking the fact that \[\lim_i \|a_q^*(T_i\xi_k) - a_q^*(\xi_k)\| = \lim_i \|a_q(T_i\xi_k) - a_q(\xi_k)\| \le (1-|q|)^{-1/2}\lim_i\|T_i \xi_k - \xi_k\| \to 0.\]
\end{proof}
\begin{cor} \label{QWEP}
The $C^{\ast}$-algebra $\mc A_{q}(\mathsf{H})$ is QWEP.
\end{cor}
\begin{proof}
We will show that $\mc A_{q}(\mathsf{H})$ is weakly cp complemented in the von Neumann algebra $\Gamma_q(\mathsf{H})$, meaning that there exists a ucp map $\Phi: \Gamma_q(\mathsf{H}) \to \left(\mc A_q(\mathsf{H})\right)^{\ast\ast}$ such that $\Phi|_{\mc A_q(\mathsf{H})} = \op{Id}$. Let $(\Phi_{i})_{i \in I}$ be the net of maps implementing at the same time the $w^{\ast}$-complete metric approximation property of $\Gamma_q(\mathsf{H})$ and the complete metric approximation property of $\mc A_q(\mathsf{H})$. Using this net, we get maps $\Phi_i: \Gamma_q(\mathsf{H}) \to \left(\mc A_q(\mathsf{H})\right)^{\ast\ast}$, as $\Phi_{i}$ maps $\Gamma_q(\mathsf{H})$ into $\mc A_q(\mathsf{H})$. There exists a cluster point of this net in the point-weak$^{\ast}$-topology and this cluster point is obviously a ucp map that is equal to identity, when restricted to $\mc A_q(\mathsf{H})$, because the net $(\Phi_{i})_{i \in I}$ converges pointwise to identity on $\mc A_q(\mathsf{H})$. Since all $q$-Araki-Woods algebras are QWEP (cf. \cite{MR2200739}) and this property descends to subalgebras that are weakly cp complemented (cf. \cite[Proposition 4.1 (ii)]{MR2072092}), we get the claimed result.
\end{proof}

\bibliographystyle{alpha}
\bibliography{Research}
\end{document}